\documentclass{amsart}
\usepackage[pagewise]{lineno}

\usepackage{amsmath,amssymb,amsthm,amsfonts,amscd,mathrsfs}
\usepackage{hyperref}
\hypersetup{
    colorlinks=true,
    linkcolor=blue
}

\usepackage[all]{xy}
\usepackage{tikz}
\usetikzlibrary{decorations.pathmorphing}
\usepackage{tikz-cd}
\usepackage{graphicx}
\usepackage{epstopdf}
\usepackage{pb-diagram}
\usepackage[]{todonotes}
\usepackage{mathtools}
\usepackage{appendix}

\DeclareMathOperator{\Tor}{\textup{Tor}}
\DeclareMathOperator{\Irr}{\textup{Irr}}
\DeclareMathOperator{\Hom}{\textup{Hom}}

\newcommand{\Ind}{\operatorname{Ind}}

\newcommand{\bdm}{\begin{displaymath}}
\newcommand{\edm}{\end{displaymath}}
\newcommand{\Sl}{\operatorname{Sl}}
\newcommand{\Spin}{\operatorname{Spin}}
\newcommand{\sign}{\operatorname{sign}}
\newcommand{\St}{\operatorname{St}}
\newcommand{\IZ}{\mathbb{Z}}
\newcommand{\R}{\mathbb{R}}
\newcommand{\Z}{\mathbb{Z}}
\newcommand{\C}{\mathbb{C}}

\newcommand{\U}{{\mathcal U}}

\newcommand{\g}{{\mathfrak{genus}}_G}

\newcommand{\Mod}[1]{\ (\mathrm{mod}\ #1)}

\newcommand{\evv}{\operatorname{ev}}

\newtheorem{defn}{Definition}[section]

\newtheorem{exam}[defn]{Example}
\newtheorem{remark}[defn]{Remark}

\theoremstyle{plain}
\newtheorem{thm}[defn]{Theorem}
\newtheorem{prop}[defn]{Proposition}
\newtheorem{lemma}[defn]{Lemma}

\title{Proper actions and decompositions in equivariant K-theory}

\author{Andr\'es Angel}
\address{Departamento de Matem\'aticas \\ Universidad de Los Andes \\ Cra. 1  No 18a-12 \\  Bogot\'a, Colombia \\ }

\email{ja.angel908@uniandes.edu.co}
\thanks{The research was carried out by the support of the grant (\#INV-2023-162-2835) from the Fondo de Investigaciones de la Facultad de Ciencias de la Universidad de los Andes. The first author acknowledges and thanks the hospitality and financial support provided by the Max Planck Institute for Mathematics in Bonn where part of this paper was  also carried out. }

\author{Edward Becerra}
\address{Departamento de Matem\'aticas\\Universidad Nacional de Colombia\\Cra. 30 cll 45 - Ciudad Universitaria\\ Bogot\'a, Colombia}
\email{esbecerrar@unal.edu.co}
\thanks{}
\author{Mario Vel\'asquez}
\address{Departamento de Matem\'aticas\\Universidad Nacional de Colombia\\Cra. 30 cll 45 - Ciudad Universitaria\\ Bogot\'a, Colombia}
\email{mavelasquezme@unal.edu.co}
\urladdr{https://sites.google.com/site/mavelasquezm/}
\thanks{}
\date{\today}
\dedicatory{The second author dedicates this paper to Hael and Heset.}
\begin{document}
	\maketitle
	\begin{abstract} In this paper, we study a decomposition of the $G$-equivariant K-theory associated to a proper $G$-space, when $G$ is a Lie group with a compact normal subgroup $A$ acting trivially. Our decomposition is related to the Mackey machine theory under certain hypotheses in order to include the topological counterpart. Our main result is a formula that decomposes the $G$-equivariant K-theory in terms of twistings related to the subgroups of $G/A$. While similar decompositions are known for compact (specifically, finite) Lie groups acting on spaces, our findings are notably more comprehensive, as they apply to a broad range of locally compact groups, including discrete, linear, and almost connected groups, whether compact or not. 
	\end{abstract}

	\section{Introduction}
	Let $G$ be a group; the (complex) representation theory of $G$  has been extensively studied because it encodes many properties associated with $G$.  Unfortunately, the representation theory of $G$ can be difficult to describe in elemental terms directly.  
	 Currently, an interesting approach exists to describe the irreducible representations of a locally compact $G$ in terms of the projective representations of a normal subgroup $A$: the theory known as the \textit{Mackey machine} \cite{Mac}. 
	
	In this paper, we use the ideas behind Mackey machine in order to generalize some of them to the context of finite dimensional $G$-vector bundles over proper $G$-spaces, with $G$ a Lie group and where the normal subgroup $A$ is compact and acts trivially on the base space. This approach results in a natural decomposition in equivariant K-theory. 
	
	More specifically, consider an extension $$1\to A\to G\to Q\to 1,$$ and let $\Irr(A)$ be the space of isomorphism classes of irreducible representations of $A$ (with the discrete topology); there is an action of $G$ (and $Q$) on $\Irr(A)$ given by conjugation (see Section \ref{sectiona}). Let $[\rho]$ be an element of $\Irr(A)$ and $G_{[\rho]}$ (respectively $Q_{[\rho]}$) be the isotropy group of $[\rho]$ under the mentioned action; we construct $S^1$-central extensions 
	$$1\to S^1\to \widetilde{Q}_{[\rho]}\to Q_{[\rho]}\to 0,$$and 	$$1\to S^1\to \widetilde{G}_{[\rho]}\to G_{[\rho]}\to 0,$$ that encode the   obstructions of lifting $\rho$ to a representation of $G_{[\rho]}$ (see Prop. \ref{notwist}). 
	
	Using both extensions, we can decompose a $G$-vector bundle, $E$, over a proper $G$-space, $X$, where $A\subseteq G$ is a compact normal subgroup acting trivially over $X$ as follows:
	
	The bundle $E$ viewed as a $A$-vector bundle can be decomposed into isotypic parts in the following way (see \cite{segal}) :
	
	$$E\cong\bigoplus_{[\rho]\in\Irr(A)}\rho\otimes \Hom_A(\rho,E).$$
	This is an isomorphism of $A$-vector bundles. Now the idea is to recover the $G$-action on the left-hand side from an \textit{ad hoc} $G$-action defined on the right hand side; this $G$ action is defined in two steps:
	
	First, we associate terms in the direct sum on the right-hand side that are related by the $G$-action over $\Irr(A)$, then the above decomposition can be expressed as
	
	$$E\cong\bigoplus_{G[\rho]\in G\setminus \Irr(A)}\left(\bigoplus_{gG_{[\rho]}\in G/G_{[\rho]}}g\cdot\rho\otimes \Hom_A(g\cdot\rho,E)\right).$$
	
	Second, we prove that it is possible to define a $G_{[\rho]}$-action over a specific $\rho\otimes\Hom_A(\rho, E)$, and then the term$$\bigoplus_{gG_{[\rho]}\in G/G_{[\rho]}}g\cdot\rho\otimes\Hom_{A}(g\cdot \rho,E)$$ can be completely recovered with its $G$-action (as an induction) if we know only one of the terms in the sum (with its $G_{[\rho]}$-action). This is the content of Thm. \ref{vectordecomp}. Now if we choose one specific $\rho\otimes\Hom_A(\rho, E)$ it can be recovered (with its $G_{[\rho]}$-action) from $\Hom_A(\rho, E)$ (with an action of some $S^1$-central extension of  $Q_{[\rho]}$). It is the content of Thm. \ref{toe}.
	
	For a Lie group $G$ and a proper $G$-space $X$, where $A$ acts trivially such that proper equivariant K-theory (denoted by $K_G(X)$) can be generated by finite-dimensional $G$-vector bundles (assumption (K)), the above procedure yields a decomposition of $K_G(X)$ in terms of equivariant (with respect to the isotropy groups $Q_{[\rho]}$, see Def. \ref{twisteddef}) twisted K-theory groups of $X$. More specifically, we have the following result:
	
	\begin{thm}\label{decomp}
		Let $G$ be a Lie group satisfying assumption (K), and $X$ be a proper $G$-space  on which the normal subgroup $A$ acts trivially. There exists a natural isomorphism 
		\begin{align*}\Psi_X:K^*_G(X)&\to\bigoplus_{[\rho]\in G\setminus\Irr(A)}{}^{\widetilde{Q}_{[\rho]}}K^*_{Q_{[\rho]}}(X)\\ [E]&\mapsto\bigoplus_{[\rho]\in G\setminus\Irr(A)}[\Hom_A(\rho,E)].\end{align*} This isomorphism is functorial on $G$-maps $X\to Y$ of proper $G$-spaces on which $A$ acts trivially. 
	\end{thm}

	The primary aim of this paper is to prove the aforementioned theorem with such generality. 
	
	
	Our results generalize the decomposition obtained in \cite{GU} for finite groups and in \cite{AGU} for compact Lie groups. For finite groups, the same  decomposition can be found in \cite{dual} using different methods. 
	
On the other hand, we also present some examples and applications of this decomposition. Specifically, we consider the antipodal action of $\Z/2$ on $\mathbb{S}^2$ then, we compute the $Q_8$-equivariant K-theory groups of $\mathbb{S}^2\times\mathbb{S}^2$ where the action is given by the canonical projection $Q_8\to \IZ_2\times\IZ_2$. We also use some computations in \cite{BaVe2016} to compute the equivariant K-theory of the classifying space for proper actions of the Steinberg group $\St(3,\Z)$, these computations correspond to the right-hand side of the Baum-Connes conjecture, which remains unproven for $\St(3,\Z)$. 

We can use the above decomposition to study proper actions of  a Lie group with only one isotropy type and obtain another natural decomposition, this is the content of Thm. \ref{oneiso}. We use that decomposition to study the equivariant K-theory groups of  actions of $SU(2)$ on simply connected 5-dimensional manifolds with only one isotropy type and orbit space $S^2$. It was pointed out to us by J. Rosenberg, that similar results for actions of compact Lie groups on locally compact spaces with only one orbit type were obtained by A. Wassermann in his PhD. Thesis at the University of Pennsylvania \cite{wassermann}, by using a different approach via $C^*$-algebras. 

This paper is organized as follows: Section \ref{sectiona} is devoted to adapt the Mackey machine \cite{Mac} to equivariant vector bundles. In Section \ref{proper$K$-theory}, we prove  Theorem \ref{decomp} describing how any $G$-equivariant vector bundle over a proper $G$-space $X$ can be obtained as the sum of inductions of twisted equivariant (with respect to certain groups) vector bundles over $X$. In Section \ref{sectionb}, we study proper actions of Lie groups with only one isotropy type and apply the decomposition of equivariant K-theory of Section \ref{sectiona}.  In Section \ref{sectionaa}, we apply the induced decomposition of equivariant K-theory to examples coming from the quaternions and the Steinberg group. We apply this decomposition to compute the equivariant K-theory of $5$-dimensional simply connected manifolds with actions of $SU(2)$ and only one isotropy type. Finally, in Appendix A, we present  additional proofs and remarks about the property $(K)$ in order to explain how this property is satisfied for a very general class of topological groups.
	
\section{Mackey machine for equivariant vector bundles and compact normal subgroup} \label{sectiona}
	
In \cite{Mac}, G. W. Mackey presented a technique to describe the unitary representations of a locally compact group, $G$, in terms of the  representations of a closed normal subgroup, $A$. This technique --called \emph{the Mackey machine}-- is a powerful tool when the structure of the unitary representations of $A$ is known. In this paper, we restrict ourselves  
to the case when $A$ is also a compact group because the Mackey machine works particularly well and the representation theory of $A$ is suitable to be applied into the setup of equivariant K-theory. Following this approach, we proceed to explain how the natural action of $G$ on the discrete set of irreducible representations of $A$ implies the existence of twisted vector bundles parametrized by a family of subgroups of $G/A$. In order to obtain useful applications to K-theory, we present here an exposition of how the Mackey machine can be generalized for equivariant vector bundles if the compact normal subgroup $A$ acts trivially on the space $X$. Throughout this section we suppose that $G$ is a Lie group and $A$ is a compact normal subgroup of $G$.
	

	Let $$1 \to A\to G\to Q\to 1$$ be the corresponding extension with $Q:=G/A$. We set $\Irr(A)$ to be the set of unitary equivalence classes of irreducible representations of $A$ endowed with the discrete topology. Let 
	$$
	\rho:A\to U(V_\rho),
	$$
	be an irreducible representation of $A$ on the vector space $V_{\rho}$ which is supposed to have an $A$-invariant metric. Let $g\in G$. We define another irreducible representation of $A$ by\begin{align*}g\cdot\rho:A&\to U(V_\rho)\\a&\mapsto\rho(g^{-1}ag).  \end{align*}resulting in a left action of $G$ over $\Irr(A)$. 
	
	Let $G_{[\rho]}$ be the isotropy group of the isomorphism class of $[\rho] \in \Irr(A)$ under the action of $G$. This means that for each $g \in G_{[\rho]}$ there is a unitary matrix $U_g$ such that, for all $a\in A$
	$$
	\rho(g^{-1}ag) = U_g^{-1}\rho(a)U_g.
	$$
First note that $A\subseteq G_{[\rho]}$. 	For a fixed $[\rho] \in \Irr(A)$ and $x\in A$, $a\cdot\rho(x)=\rho(x^{-1}ax)=\rho(x^{-1})\rho(a)\rho(x)=\rho(a)$ for all $a\in A$. Then taking $U_x=\rho(x)$ we see that $x\in G_{[\rho]}$. Let us define $Q_{[\rho]}=G_{[\rho]}/A$. Therefore, there is an extension:
	$$
	1 \to A \stackrel{\iota  }{\rightarrow} G_{[\rho]} \stackrel{}{\rightarrow} Q_{[\rho]} \rightarrow  1.
	$$
	
	\begin{remark}
		If $A$ is central in $G$, then the action of $G$ on $\Irr(A)$ is trivial, $G_{[\rho]}=G$ and $Q_{[\rho]}=Q$.
	\end{remark}

	For any pair $g,h\in G_{[\rho]}$, and for all $a\in A$:$$U_{(gh)^{-1}}\rho(a)U_{gh}=\textcolor{black}{\rho((gh)^{-1}a(gh))=U_h^{-1}\rho(g^{-1}ag)U_h=}U_{h^{-1}}U_{g^{-1}}\rho(a)U_{g}U_{h}.$$By straightforward calculations, the equation above can be rewritten as:$$U_{gh}U_{h^{-1}}U_{g^{-1}}\rho(a)=\rho(a)U_{gh}U_{h^{-1}}U_{g^{-1}}.$$
	As $\rho$ is supposed to be irreducible, by Schur's lemma $U_{gh}$ must be a multiple of $U_{g}U_{h}$ and therefore there is a well defined homomorphism  $\Upsilon:G_{[\rho]}\to PU(V_\rho)$ that extends $\rho: A \rightarrow U(V_\rho)$. Summarizing, we have the following result.
	\begin{lemma}\label{lema41}
		There is a unique homomorphism $\Upsilon:G_{[\rho]}\to PU(V_\rho)$ such that the following diagram commutes
		$$\xymatrix{A\ar[r]^\iota \ar[d]^\rho&G_{[\rho]}\ar[d]^{\Upsilon}\\U(V_\rho)\ar[r]^p&PU(V_\rho).}$$
	\end{lemma}

	Recall the canonical central extension:
	\[
	1  \to S^{1} \stackrel{i}{\rightarrow} U(V_{\rho})\stackrel{p}{\rightarrow} PU(V_{\rho})\to 1
	\]
	where $PU(V_\rho) = U(V_\rho))/S^1$ and $p(u)$ is the conjugation by $u$.

	Define the Lie group  $\widetilde{G}_{[\rho]} : = G_{[\rho]} \times_{PU(V_\rho)} U(V_\rho) = \{ (g,u) \mid {\Upsilon}(g) = p(u)\}$. Note that $\ker(\pi_1)=1\times S^1\subseteq G_{[\rho]} \times U(V_\rho)$. Then there is a central extension
	\[
	1 \to S^{1} \stackrel{i}{\rightarrow} \widetilde{G}_{[\rho]} \stackrel{\pi_1}{\rightarrow}  G_{[\rho]} \to 1 
	.\]

	There is a commutative diagram of $S^1$-central extensions of Lie groups
	$$\xymatrix{1\ar[r]&S^1\ar[r]\ar[d]^{id}&\widetilde{G}_{[\rho]}\ar[r]^{\pi_1} \ar[d]^{\widetilde{\Upsilon}}&G_{[\rho]}\ar[r]\ar[d]^{\Upsilon}&1\\1\ar[r]&S^1\ar[r]&U(V_\rho)\ar[r]&PU(V_\rho)\ar[r]&1}$$
	
	Where $\widetilde{\Upsilon}$ is the projection $\pi_2:\widetilde{G}_{[\rho]}\to U(V_{\rho})$. 
	The homomorphism $ \iota \times \rho :  A \rightarrow G_{[\rho]} \times U(V_\rho)$ has the image contained in $\widetilde{G}_{[\rho]}$. Let us see that $(\iota \times \rho) (A) $ is normal in $\widetilde{G}_{[\rho]}$. To prove this we need to check that if $a \in A$ and $(g,u)  \in G_{[\rho]} \times_{PU(V_\rho)} U(V_\rho) $ then 
	$$
	(g,u) (\iota \times \rho ) (a)  (g,u)^{-1} \in (\iota \times \rho) (A). 
	$$
	We have
	$$
	(g,u) (\iota \times \rho ) (a)  (g,u)^{-1} = ( gag^{-1},u \rho(a) u^{-1}).
	$$
	Since $(g,u) \in G_{[\rho]} \times_{PU(V_\rho)} U(V_\rho)$, we know that $\Upsilon(g) = p(u)$, i.e 
	$$
	\rho(gag^{-1} ) = u \rho(a) u^{-1}
	$$
	and therefore
	$$
	(g,u) (\iota \times \rho ) (a)  (g,u)^{-1} =  (\iota \times \rho ) ( gag^{-1})
	.$$This proves that $(\iota\times\rho)(A)$ is normal in $\widetilde{G}_{[\rho]}$.

	Thus the quotient $\widetilde{G}_{[\rho]}/ (\iota \times \rho) (A)$
	is a Lie group. We denote this Lie group by $\widetilde{Q}_{[\rho]}$ and there is an exact sequence
	$$
	1 \to A \stackrel{\iota \times \rho }{\rightarrow} \widetilde{G}_{[\rho]} \stackrel{}{\rightarrow} \widetilde{Q}_{[\rho]} \rightarrow  1.
	$$
	
	We want to see now that there is a central extension
	\begin{equation} \label{centralexquot}
	1 \to S^{1} \stackrel{\widetilde{\iota}}{\rightarrow} \widetilde{Q}_{[\rho]} \xrightarrow{\widetilde{\pi}} Q_{[\rho]} \to 1
	\end{equation}
	
	where $\widetilde{\iota} : S^1 \rightarrow \widetilde{G}_{[\rho]}/ (\iota \times \rho) (A)$ is induced by  $1 \times S^1 \subseteq G_{[\rho]} \times_{PU(V_\rho)} U(V_\rho)$.

	Note that $ (\{e \} \times S^1) \cap  (\iota \times \rho) (A) = \{(e,id_{U(V_\rho)} )\}$ and therefore 
	$$
	S^{1} \stackrel{i}{\rightarrow} \widetilde{G}_{[\rho]} \rightarrow \widetilde{G}_{[\rho]}/ (\iota \times \rho) (A)
	$$
	is injective and gives a map $\widetilde{\iota} : S^1 \rightarrow \widetilde{G}_{[\rho]}/ (\iota \times \rho) (A)=\widetilde{Q}_{[\rho]} $. On the other hand the map $\widetilde{\pi}$ is defined as $(g,u)(\iota\times\rho)(A)\mapsto gA$. This map is well defined and its kernel is precisely the image of $\widetilde{\iota}$. Then we have the desired exact sequence.
	
	
	$$
	$$

	There is a commutative diagram of extensions of Lie groups
	$$\xymatrix{1 \ar[r]&A\ar[r]\ar[d]^{id}& \widetilde{G}_{[\rho]}\ar[r] \ar[d]^{}& \widetilde{Q}_{[\rho]} \ar[r]\ar[d]& 1 \\ 1 
		\ar[r]&A\ar[r] &G_{[\rho]}\ar[r]^{\pi_1} &Q_{[\rho]} \ar[r] & 1 
	}$$
	and also a commutative diagram of $S^1$-central extensions of Lie groups
	$$\xymatrix{1 \ar[r]&S^1\ar[r]\ar[d]^{id}&\widetilde{G}_{[\rho]}\ar[r]^{\pi_1} \ar[d]^{}&G_{[\rho]}\ar[r]\ar[d]& 1 \\
		1\ar[r] &  S^{1}  \ar[r]^{\widetilde{i}} &  \widetilde{Q}_{[\rho]} \ar[r] & Q_{[\rho]} \ar[r] & 1.
	}$$
	
	We will now examine the conditions under which these extensions are trivial.
	
	\begin{remark}\label{abelian-kernel}
		If $\rho$ is a $1$-dimensional irreducible representation of $A$, then the $S^1$-extension 
		\[
		1 \to S^{1} \stackrel{i}{\rightarrow} \widetilde{G}_{[\rho]} \stackrel{\pi_1}{\rightarrow}  G_{[\rho]} \to 1, \]
		is trivial. This is because in this case $PU(V_\rho)=1$.

	\end{remark}
	
	Exactly as in \cite{AGU}, we have:
	
	\begin{prop} \label{notwist}
		The irreducible representation $\rho$ can be extended to $G_{[\rho]}$ if, and only if, the $S^1$-central extension of Lie groups 
		$$1 \to S^1\to \widetilde{Q}_{[\rho]}\to Q_{[\rho]} \to 1$$
		is trivial.
	\end{prop}\begin{proof}
        Suppose the above central extension is trivial; consequently, we have a section $$\tilde{s}:Q_{[\rho]}\to \widetilde{Q}_{[\rho]}.$$Let $g\in G_{[\rho]}$ and  $\tilde{s}(gA)=(g',u)(\iota\times \rho)(A)$ be the image of the class $gA$. First, observe that $g'A=gA$. We may then assume that $(g',u)(\iota\times \rho)(A)=(g,u_{g})(\iota\times \rho)(A)$, with $u_g$ uniquely determined by $g$. Let us define a section \begin{align*}s: G_{[\rho]}&\to \widetilde{G}_{[\rho]}\\
  g&\to (g,u_g).
  \end{align*} Then $$1\to S^1\to\widetilde{G}_{[\rho]}\to G_{[\rho]}\to1$$ is also trivial.  Since $\widetilde{G}_{[\rho]}\cong G_{[\rho]}\times S^1$ as Lie groups, it means that there is a Lie group homomorphism $\sigma:G_{[\rho]}\to\widetilde{G}_{[\rho]}$ that is a right inverse of the quotient map. Let $\widetilde{\Upsilon}:\widetilde{G}_{[\rho]}\to U(V_\rho)$ be the associated representation to ${\Upsilon}:G_{[\rho]}\to PU(V_\rho)$. $\widetilde{\Upsilon}\circ\sigma:G_{[\rho]}\to U(V_\rho)$ is an extension of $\rho$ because for $a\in A$\begin{align*}
		\widetilde{\Upsilon}(\sigma(i(a)))&=\widetilde{\Upsilon}(\widetilde{i}(a))\\ &=\rho(a),
		\end{align*}
		where the last equality follows by the definition of $\widetilde{i}$.
		
		On the other hand, if $\widetilde{\rho}:G_{[\rho]}\to U(V_\rho)$ is an extension of $\rho$, we can define the morphism ${\Upsilon}$ in Lemma \ref{lema41} explicitely by $\pi\circ\widetilde{\rho}:G_{[\rho]}\to PU(V_\rho)$. Recall that $$\widetilde{G}_{[\rho]}=\{(g,\chi)\in G_{[\rho]}\times U(V_\rho)\mid \pi(\widetilde{\rho}(g))=\pi(\chi)\}.$$ Define \begin{align*}\sigma:G_{[\rho]}&\to\widetilde{G}_{[\rho]}\\g&\mapsto(g,\widetilde{\rho}(g)).\end{align*}
		Then $\widetilde{G}_{[\rho]}\cong G_{[\rho]}\times S^1$ as Lie groups. It implies that $\widetilde{Q}_{[\rho]}\cong Q_{[\rho]}\times S^1$ as Lie groups and thus the extension is trivial.
	\end{proof}
	Note that the $S^1$-central does only depend on the class of isomorphism $[\rho]$, then we can state the following:
	\begin{defn}
	    Let $\rho$ be an irreducible representation of $A$. We refer to the sequence in Proposition \ref{notwist} as the \emph{twisting} associated to $\rho$.
	\end{defn}

	\subsubsection{Semidirect products with abelian normal subgroup}
	
	Now we will study the case of a semidirect product. As it is usual in representation theory, for a  group $G$, we denote by $$\widehat{G}=\{f:G\to \mathbb{C}^\times \mid f\text{ homomorphism }\},$$ the space of characters of $G$ (with the compact open topology).
	
	Let $A \mathrel{\unlhd} G$ be a normal subgroup of $G$. Then the group $G$ acts on $\widehat{A}$ by taking a character $\chi$ of $A$, and defining $(g \cdot \chi)(a) = \chi(g^{-1}ag)$ where $g\in G$. For a character $\gamma \in \widehat{G}$, consider its restriction to the subgroup $A\subset G$ as an element $\gamma|_A\in \widehat{A}$. If $\chi=\gamma|_A$ for some $\gamma\in\widehat{G}$, since $\C^\times$ is abelian, for every $g\in G$, $g \cdot \chi =\chi$. Thus the restriction homomorphism $\widehat{G} \rightarrow \widehat{A}$ has its image in  $$\widehat{A}^G=\{\alpha\in\widehat{A}\mid g\cdot \alpha=\alpha\},$$ the subgroup of fixed points by the action of $G$.
	\begin{prop}\label{propoaux}
		If $G = A \rtimes Q$, then there is a group isomorphism
		$$
		\Gamma: \widehat{G} \rightarrow {\widehat{A}}^Q \times \widehat{Q},
		$$given by restriction.
			\end{prop}
		\begin{proof}
			$\Gamma$ is injective because it is the restriction to both $A$ and $Q$, and because $G=AQ$. To see that it is surjective, take $\alpha \in {\widehat{A}}^Q $ and $\beta \in \widehat{Q}$ and consider the product $\alpha \rtimes \beta$ defined for $g\in G$ as follows: there is a unique pair $(a,q)\in A\times Q$ such that $g=aq$; define  $\alpha\rtimes\beta(g)=\alpha(a)\beta(q)$. To see this is a character, note:
			$$
			\begin{aligned} 
   a_1 q_1 a_2 q_2&=a_1 q_1 a_2 q_1 ^ { - 1 } q_1 q_2\\
   \alpha \left( a_1 q_1 a_2 q_1^ { - 1 } \right)&=\alpha ( a_1 ) \alpha \left( q_1 a_2 q_1 ^ { - 1 } \right)\\
   &=\alpha ( a_1 ) \alpha \left( a_2 \right)
			\end{aligned}
			$$Clearly $\Gamma(\alpha\rtimes\beta)=(\alpha,\beta)$.
		\end{proof}
	
	\begin{prop}
		Let $A$ be an abelian group and $G = A \rtimes Q$. Then for every irreducible representation $\chi$ of $A$, the $S^1$-central extension of Lie groups 
		$$1 \to S^1\to \widetilde{Q}_{\chi}\to Q_{\chi} \to 1$$
		is trivial.\end{prop}
		\begin{proof}We want to apply Proposition \ref{notwist}.
			Suppose $A$ is abelian, and take $\chi \in \widehat{A}$. For $g\in G_{\chi}$ 
			$$
			\chi(g^{-1}ag) =U_g^{-1}\chi(a)U_g = \chi(a)
			$$
			since $U_g$, $\chi(a)$ and $U_{g}^{-1}$ are elements of  $S^1$; therefore every character of $A$ is invariant under $G_{\chi}$. Also, $G_{\chi} = A \rtimes Q_{\chi}$ since clearly $G_{\chi} \supseteq A \rtimes Q_{\chi}$ and for $(\widetilde{a},q) \in G_{\chi}$
			$$
			\chi(a) = \chi((\widetilde{a}q)^{-1}a (\widetilde{a}q)) = \chi(q^{-1}\widetilde{a}^{-1}a \widetilde{a}q) = \chi(q^{-1}aq)
			$$
			and therefore $q \in Q_{\chi}$.
			
			Taking the trivial character of $Q_{\chi}$ gives in a character of $\widehat{A}^{G_\chi}\rtimes \widehat{Q}_\chi=\widehat{A}\rtimes \widehat{Q}_\chi$ and then, by proposition \ref{propoaux}, we have a character of $G_{\chi}$ that extends $\chi$, therefore by  Proposition \ref{notwist} we are done.

		\end{proof}
	
	Next, we describe the Mackey machine for equivariant vector bundles when the normal subgroup is compact.

	\subsection{Decomposition of $G$-vector bundles}
	Let $E$ be a $G$-vector bundle over a $G$-space $X$, and let $A$ be a compact normal subgroup of $G$ acting trivially on $X$. Since $A$ is compact we can assume that $E$ is endowed with an $A$-invariant Hermitian metric. Each fiber is therefore a unitary representation of $A$. By a result in \cite{segal} we know that the restriction $E\mid_A$ considered as an $A$-vector bundle can be decomposed into isotypic parts  $$E\mid_A \quad \cong\bigoplus_{[\rho]\in \Irr(A)}  \rho \otimes \Hom_A(\rho,E),$$where $\rho$ is considered as the trivial $A$-bundle $X \times V_\rho \rightarrow X$ and $\Hom_A(E_1,E_2)$ is the bundle of $A$-equivariant maps between two $A$-bundles $E_1,E_2$ endowed with the trivial $A$-action. Let $E_\rho=\rho\otimes\Hom_A(\rho,E)$ and call it the \emph{$\rho$-isotypic part of $E$}.
	\begin{defn}
		Let $X$ be a $G$-space, where the action of $A$ is trivial and $\rho:A\to U(V)$ is an irreducible $A$-representation. A $G_{[\rho]}$-vector bundle $p:E\to X$
		is $(G_{[\rho]},\rho)$-isotypic if each fiber $E_x$ is an $A$-representation isomorphic to an integer multiple of $\rho$.
	\end{defn}
	
	
	Another way to say that each fiber $E_x$ is an $A$-representation isomorphic to a multiple of $\rho$ is that the map
	\begin{align*}
	\beta:E_\rho&\to E\\
	v\otimes f&\mapsto f(v)
	\end{align*}
	is an isomorphism of $A$-vector bundles.\\
	
	As in Thm. 2.3 in \cite{AGU}, $(G_{[\rho]},\rho)$-isotypic vector bundles are in correspondence with $\widetilde{Q}_{[\rho]}$-vector bundles on which $S^1$-acts by multiplication by inverses. The  proof of \cite{AGU} did not depend on the compactness of $G$. We include it below for completeness.
	
	\begin{thm}\label{toe}
		Let $X$ be a $G$-space where $A$ acts trivially on $X$.  There is a natural one-to-one correspondence between isomorphism classes of $(G_{[\rho]},\rho)$-isotypic vector bundles over $X$ and isomorphism classes of $\widetilde{Q}_{[\rho]}$-equivariant vector bundles over $X$ on which $S^1$-acts by multiplication by inverses. 
		
	\end{thm}
	\begin{proof}
		Suppose that $p:E\to X$ is a $(G_{[\rho]},\rho)$-isotypic vector bundle over $X$. Then $\Hom_A(\rho,E)$ is a nonzero complex vector bundle. 
		
		 We keep the notation in Lemma \ref{lema41} and endow $\Hom_A(\rho,E)$  with a linear $\widetilde{G}_{[\rho]}$-action in which $\pi_1^{-1}(A)$ acts trivially in the following way.
		Note that $\widetilde{G}_{[\rho]}$ acts on $\rho$ via the map $\widetilde{\Upsilon}:\widetilde{G}_{[\rho]}\to U(V_\rho)$  and also acts on $E$ via the natural projection $\pi_1:\widetilde{G}_{[\rho]}\to G_{[\rho]}$, then $\widetilde{G}_{[\rho]}$ acts in a natural way on $\Hom_A(\rho,E)$ by conjugation. More explicitly, let $\widetilde{g}\in \widetilde{G}_{[\rho]}$, $\varphi\in \Hom_A(\rho,E)$ and $v\in\rho$, the action is defined as
		$$(\widetilde{g}\bullet\phi)(v)=\pi_1(\widetilde{g})\phi\left(\left(\widetilde{\Upsilon}\left(\widetilde{g}\right)\right)^{-1}v\right)$$  This action is continuous. Moreover as we are considering $A$-equivariant homomorphisms it is clear that $\pi_1^{-1}(A)$ acts trivially, thus we get an action of $\widetilde{Q}_{[\rho]}$.		
		
		From the definition of the action it is clear that $S^1=\ker(\pi_1)$ acts by multiplication by  inverses.
		
		
		Let us consider the transformation$$[E]\mapsto[\Hom_A(\rho,E)]$$from the set isomorphism classes of $(G_{[\rho]},\rho)$-isotypic vector bundles over $X$ to the set of isomorphism classes of $\widetilde{Q}_{[\rho]}$-equivariant vector bundles over $X$  for which $S^1$-acts by multiplication by inverses. We will see that this transformation is a bijective correspondence. For this we will define an inverse.
		
		Let $p:F\to X$ be a $\widetilde{Q}_{[\rho]}$-vector bundle on which $S^1$-acts by multiplication by its inverses. Consider the vector bundle $\rho\otimes F$.
		
	Note that $\widetilde{G}_{[\rho]}$ acts on $F$ via the quotient map and acts on $\rho$ via the map $\Upsilon$, then $\widetilde{G}_{[\rho]}$ acts on the tensor $\rho\otimes F$. More explicitly, let $\widetilde{g}\in \widetilde{G}_{[\rho]}$ and $v\otimes f\in\rho\otimes F$
		
		$$\widetilde{g}\cdot(v\otimes f)=\left(\widetilde{\Upsilon}\left(\widetilde{g}\right)v\right)\otimes(\widetilde{g}\tau^{-1}(A)\cdot f)$$ where $\widetilde{g}\pi_1^{-1}(A)$ is the coset corresponding to $\widetilde{g}$ in $\widetilde{Q}_{[\rho]}.$ This action is continuous. Moreover, if $\lambda\in S^1=\ker(\tau)$,
		$$\lambda\cdot(v\otimes f)=\left(\widetilde{\Upsilon}(\lambda)v\right)\otimes\left(\left(\lambda\pi_1^{-1}(A)\right)\cdot f\right)=\lambda v\otimes\lambda^{-1}f=v\otimes f.$$
		Thus $\ker(\pi_1)$ acts trivially, and then we have a well defined action of $G_{[\rho]}=\widetilde{G}_{[\rho]}/S^1$.
		
		Now we will see that  $\rho\otimes F$ is a $(G_{[\rho]},\rho)$-isotypic vector bundle. 
		
		Let $a\in A$, and let $\widetilde{a}\in \pi_1^{-1}(a)$,
		$$a\cdot(v\otimes f)=\widetilde{a}(v\otimes f)=\left(\widetilde{\Upsilon}\left(\widetilde{a}\right)v\right)\otimes\left(\widetilde{a}\pi_1^{-1}(A)\cdot f\right)=\rho(a)v\otimes f.$$
		This implies that $\rho\otimes F$ is a $(G_{[\rho]},\rho)$-isotypic vector bundle.
		
		Now consider the transformation
		
		$$[F]\mapsto[\rho\otimes F]$$
		from the set of isomorphism classes of $\widetilde{Q}_{[\rho]}$-equivariant vector bundles over $X$  for which $S^1$ acts by multiplication by inverses to the set of isomorphism classes of  $(G_{[\rho]},\rho)$-isotypic vector bundles over $X$. We will see that this transformation is the inverse of the transformation defined previously.
		
		
		Note that \begin{align*}
		\operatorname{ev}:\rho\otimes\Hom_A(\rho,E)&\to E\\v\otimes\phi&\mapsto\phi(v)
		\end{align*}
		is an isomorphism of vector bundles. We will see that $\evv$ is $G_{[\rho]}$-equivariant.
		
		Let $g\in G_{[\rho]}$ and $\widetilde{g}\in\widetilde{G}_{[\rho]}$ with $\pi_1(\widetilde{g})=g$.
		\begin{align*}
		\evv(g\cdot(v\otimes\phi))&=\evv\left(\widetilde{g}\cdot(v\otimes\phi\right))\\&=\evv\left(\left(\widetilde{\Upsilon}(\widetilde{g})v\right)\otimes\left(\widetilde{g}\bullet\phi\right)\right)\\&=\left(\widetilde{g}\bullet\phi\right)\left(\widetilde{\Upsilon}\left(\widetilde{g}\right)v\right)\\&=g\phi\left(\widetilde{\Upsilon}\left(\widetilde{g}\right)^{-1}\widetilde{\Upsilon}\left(\widetilde{g}\right)v\right)\\&=g\phi(v).
		\end{align*}
		Then $\rho\otimes\Hom_A(\rho,E)\cong E$ as $(G_{[\rho]},\rho)$-isotypic vector bundles.
		
		Finally let us prove that  $F\cong\Hom_A(\rho,\rho\otimes F)$ as $\widetilde{Q}_{[\rho]}$-equivariant vector bundles over $X$  for which $S^1$-acts by multiplication by inverses.
		
		As $A$ acts trivially on $F$ it is the case that \begin{align*}
		F&\to\Hom_A(\rho,\rho\otimes F)\\f&\mapsto (v\mapsto v\otimes f)
		\end{align*}
		is an isomorphism of $\widetilde{Q}_{[\rho]}$-vector bundles. Therefore we have a bijective corespondence and the results follows.
	\end{proof}
	
	Now using the above theorem we will decompose every $G$-vector bundle over an $A$-trivial $G$-space in terms of equivariant $A$-vector bundles, which turn out to be twisted $Q_{[\rho]}$-equivariant vector bundles (see Theorem \ref{decomp} below). To do that, we need to recall the construction of the induced vector bundle.
	
	\subsection{Induction}
	
	Let $H\subseteq G$ be a subgroup of $G$ of finite index and $X$ be a $H$-space, $G$ is an $H$-space via the right multiplication then the orbit space $$G\times_HX=(G\times X)/H,$$is naturally a $G$-space. The $G$-action is defined as $$g\cdot [g',x]=[gg',x]\text{ for every }g,g'\in G\text{ and  every } x\in X.$$Suppose that $X$ is a $G$-space, then the map \begin{align*}
	    (G\times X)/H&\xrightarrow{\eta} G/H\times X\\
     [g,x]&\mapsto (gH,gx)
	\end{align*} is a $G$-equivariant homeomorphism. 
 
 \begin{defn}
     Let $E\xrightarrow{\pi} X$ be a $H$-vector bundle over a $G$-space $X$. The map 
 $$G\times_H E\xrightarrow{id\times \pi}G\times_H X$$ is naturally a $G$-vector bundle over $G\times_HX$, the induction from $H$ to $G$ of $E$ is the $G$-vector bundle over $X$ 
 \begin{align*}
    \Ind_H^G(E)=G\times_H E\xrightarrow{(id\times\pi)\circ\eta\circ\pi_2}X
 \end{align*}
 
 \end{defn}
 
	
	
 The main  property of $\Ind_H^G(E)$ is the following theorem.
	
	\begin{thm}[Frobenius reciprocity] \label{inductionprop}
			Let $G$ be a Lie group, and let $H\subseteq G$ be a subgroup of finite index. Let $X$ be a $G$-space, and let $E$ be an $H$-vector bundle over $X$. There is a unique  $G$-vector bundle $\Ind_H^G(E)$ over $X$, up to isomorphism of $G$-vector bundles, such that for every $G$-vector bundle $F$ over $X$ we have a natural identification$$\Hom_G(\Ind_H^G(E),F)\cong\Hom_H(E,F).$$
	\end{thm}
	\begin{proof} The argument in Thm. 2.9 in \cite{CRV} applies to any Lie group. 
	\end{proof}

	Let $E$ be a $G$-vector bundle over $X$, as $E$ is finite-dimensional there are only a finite number of isomorphism classes $[\rho]$ such that $$\Hom_A(\rho,E)\neq0.$$ Moreover if $\Hom_A(\rho,E)\neq0$, for every $g\in G$ we have $\Hom_A(g\cdot\rho,E)\neq0.$ 
 Hence, if $[\rho]$ appears in $E$, then every element in its orbit appears in $E$. This implies that $G_{[\rho]}$ has finite index in $G$. Therefore we can apply the induction defined previously. We have the following result.
	
	\begin{thm}\label{vectordecomp}Let $G$ be a Lie group and $X$ be a $G$-space. Suppose that $A$ is a compact normal subgroup of $G$ acting trivially over $X$. Let $p:E\to X$ be a finite-dimensional $G$-vector bundle,. Then we have a natural (on $X$) isomorphism of $G$-vector bundles
		$$\bigoplus_{[\rho]\in G\setminus \Irr(A)}\Ind_{G_{[\rho]}}^{G}\left(\rho\otimes\Hom_A(\rho,E)\right)\to E.$$where $[\rho]$ runs over representatives of the orbits of the $G$-action on the set of isomorphism classes of representations of $A$.
	\end{thm}
	\begin{proof} 
 We have a $G_{[\rho]}$-equivariant map
 \begin{align*}\Hom_A(\rho,E)&\xrightarrow{\evv} E\\(v,\phi)&\mapsto (\phi(v)) 
 \end{align*}
    that by Thm. \ref{inductionprop}  gives a $G$-equivariant map  \begin{align}\label{Segal}\Ind_{G_{[\rho]}}^{G}\left(\rho\otimes\Hom_A(\rho,E)\right)\to E \end{align} natural in $X$. Moreover, the map (\ref{Segal}) is precisely the inverse of the isomorphism in  Prop. 2.2 in \cite{segal}, then we obtain that it is an isomomorphism of $G$-vector bundles natural in $X$.
	\end{proof}
	The previous theorems applied to $X=*$ are the usual Mackey machine which describes how to get the finite-dimensional unitary representations of $G$ in terms of irreducible representations of $A$ and $\widetilde{Q}_{[\rho]}$-twisted representations of $Q_{[\rho]}$: 
	\begin{enumerate}
		\item Every finite-dimensional  representation of $G$ is the sum of induced representations of $(G_{[\rho]},\rho)$-isotypic ones (Theorem \ref{vectordecomp}).
		\item $(G_{[\rho]},\rho)$-isotypic representations can be obtained by tensoring the extension of $\rho$ to $G_{[\rho]}$ with  $\widetilde{Q}_{[\rho]}$-twisted representations of $Q_{[\rho]}$ (Theorem \ref{toe}).
	\end{enumerate}

	\section{Decomposition in equivariant K-theory (proper case)}\label{proper$K$-theory}
	For a general Lie  group $G$, we extend the definition in \cite{lo} of  equivariant K-theory for proper actions, to twisted equivariant K-theory for the twisting relevant to our decompositions. See also \cite{cantat}. 
	
	For general Lie groups we work on the category of locally compact Hausdorff second countable spaces.  If the group is discrete or compact we can also work on the category of proper $G$-CW complexes.

	\begin{defn}\label{twisteddef}
		Let $Q$ be a Lie group acting properly on a  $Q$-space $X$ and 
		\begin{equation}\label{centralexten}
		1\to S^1\to\widetilde{Q}\to Q\to 1
		\end{equation}
		be a $S^1$-central extension. Define  ${}^{\widetilde{Q}}\mathbb{K}_Q^0(X)$ as the Grothendieck group of the monoid of isomorphism classes of $\widetilde{Q}$-vector bundles over $X$ where $S^1$ acts by multiplication by inverses. Define for $n\geq0$ $${}^{\widetilde{Q}}\mathbb{K}_Q^{-n}(X)=\ker\left({}^{\widetilde{Q}}\mathbb{K}_Q(X\times S^n)\xrightarrow{i^*}{}^{\widetilde{Q}}\mathbb{K}_Q(X)\right),$$ where $Q$ acts trivially on $S^n$ and $i(x)=(x,e_1)$ with $e_1$ the first element in the canonical basis of $\mathbb{R}^{n+1}$.  
		For any proper $Q$-pair $(X,A)$,  set$${}^{\widetilde{Q}}\mathbb{K}_Q^{-n}(X,A)=\ker\left({}^{\widetilde{Q}}\mathbb{K}_Q^{-n}(X\cup_AX )\xrightarrow{i_2^*}{}^{\widetilde{Q}}\mathbb{K}_Q^{-n}(X)\right).$$ Where $i_2:X\to X\cup_AX$ is the map induced by the inclusion of $X$ in the second copy of $X\amalg X$. Note that, when $\widetilde{Q}\cong S^1\times Q$, a $\widetilde{Q}$-vector bundle over $X$ where $S^1$ acts by multiplication by inverses is the same as a $Q$-vector bundle; then we have a canonical identification
		$\mathbb{K}_Q^{*}(X) \cong {}^{\widetilde{Q}}\mathbb{K}_Q^{*}(X)  \text{ and  }\mathbb{K}_Q^{*}(X,A) = {}^{\widetilde{Q}}\mathbb{K}_Q^{*}(X,A)$.
	\end{defn}
	Recall that a functor from a subcategory of $Q$-spaces to abelian groups is a $Q$ cohomology theory if it satisfies the following axioms:
			\begin{enumerate}
			\item $Q$-homotopy invariance.
			\item Long exact sequence of a pair.
			\item Excision.
			\item Disjoint union axiom.
			\end{enumerate}  
	For more details on the axioms see \cite{luckcoh}. Suppose that the Lie group $Q$ satisfies the following assumption:
	\begin{equation}
	\tag{K}\label{eq:A}
	\parbox{\dimexpr\linewidth-4em}{%
		\strut%
		\emph{For every compact normal subgroup $A$  and for each isomorphism class of an irreducible representation $[\rho] \in \Irr(A)$, 
			the functors ${}^{\widetilde{Q}_{[\rho]}}\mathbb{K}_{Q_{[\rho]}}^*(-)$ define a $\Z/2$-graded $Q_{[\rho]}$-cohomology  theories. 
		}\strut
	}\end{equation} 
	Note that by taking $A$ the trivial group, we have $Q_{[\rho]}=G_{[\rho]}=G$ and $\widetilde{Q}_{[\rho]} = S^1 \times G $, hence we are requiring in particular, that $\mathbb{K}_G^{*}(-)$ is a $G$-cohomology theory.

	When $Q$ satisfies the above assumption we write ${}^{\widetilde{Q}}K_Q^{-n}(X,A)$ for ${}^{\widetilde{Q}}\mathbb{K}_Q^{-n}(X,A).$
	\begin{remark}\label{nota1}
		From \cite{lo} and \cite{dwyer}, a discrete group $G$ satisfies assumption (K) (in the category of finite, proper $G$-CW complexes). Additionally, any compact Lie group also satisfies assumption (K) (in the category of locally compact Hausdorff second countable $G$-spaces). 
        \end{remark}  
		In the appendix  we will build upon the work of \cite{Phil2}  to prove that  almost connected and linear Lie groups satisfy assumption $(K)$ (in the category of locally compact Hausdorff second countable $G$-spaces). However there are Lie groups where assumption (K) fails to be true; see   \cite{Phil} and   \cite{lo}. 
	\
	
	For a general Lie group, there is an induction scheme: if  $H$ is a subgroup of $G$ and $Y$ is a proper $H$-space, then
	$$
	\mathbb{K}_G^*(G \times_H Y) \cong \mathbb{K}_H^*(Y)
	$$
	(see lemma 3.4 of \cite{lo}). A similar result applies to free actions, analogue to Lemma 3.5 of \cite{lo}.
	
	Let $G$ be a Lie group acting properly on a  $G$-space $X$ and 
	\begin{equation}
	1\to S^1\to\widetilde{G}\to G\to 1
	\end{equation}
	be a $S^1$-central extension. When $G$ acts freely on $X$,
	$$
	{}^{\widetilde{G}} \mathbb{K}^*_{G}(X) \cong {}^{\alpha} K^*(X/G)
	$$
	where the right hand side denotes the \emph{classical} twisted K-theory (defined for example in \cite{AS}), and  $\alpha$ is a cohomology class induced from the central extension (\ref{centralexten}). This is done for finite groups in Proposition 3.8 of \cite{GU}, and for proper actions of Lie groups on manifolds using $C^*$-algebras in Proposition 3.6 of \cite{twistedstacks}. 
	The twisting $\alpha$ can be described as follows: for the central extension
	$$
	1 \to S^{1} \stackrel{\widetilde{\iota}}{\rightarrow} \widetilde{G} \rightarrow G \to 1
	$$
	consider the cohomology class $\alpha_{\widetilde{G}} \in H^2(BG;S^1)$  that classifies it (see Thm. 10 in \cite{moore}). We have the fibration
	$
	p : EG \times_G X \rightarrow BG,
	$
	and since the action is free, there exists a homotopy equivalence
	$
	EG \times_G X \rightarrow X/G
	$
	giving a commutative diagram
	$$
	\xymatrix{
		EG \times_G X \ar[r]^{p} \ar[d]^{\simeq} & BG \\
		X/G \ar[ur]
	}
	$$
	Consider the cohomology class $\alpha \in H^2(X/G;S^1) \cong H^2_G(X;S^1)$ given by $p^*(\alpha_{\widetilde{G}})$.

	\begin{thm}\label{decompintro}
		Let $G$ be a Lie group satisfying assumption (K), and $X$ be a proper $G$-space  on which the normal subgroup $A$ acts trivially. There is a natural isomorphism 
		\begin{align*}\Psi_X:K^*_G(X)&\to\bigoplus_{[\rho]\in G\setminus\Irr(A)}{}^{\widetilde{Q}_{[\rho]}}K^*_{Q_{[\rho]}}(X)\\ [E]&\mapsto\bigoplus_{[\rho]\in G\setminus\Irr(A)}\Hom_A(\rho,E).\end{align*} This isomorphism is functorial on $G$-maps $X\to Y$ of proper $G$-spaces on which $A$ acts trivially. 
	\end{thm}
	\begin{proof}
		It is a straightforward application of Theorems \ref{toe} and \ref{vectordecomp}.     
	\end{proof}

	\section{Structure of actions with only one isotropy type}\label{sectionb}

	When the action of a compact Lie group on a space $X$ is not free but has only one isotropy type $(G/H)$, with $H$  a closed subgroup of $G$, we have that $X \rightarrow X/G$ is a fiber bundle with fiber $G/H$ and structure group $NH/H$.
	

		$$
		NH/H \rightarrow X^H \rightarrow X/G.
		$$
	
	In fact, a $G$-space with only one orbit type  is  the $G$-induction of the associated fiber bundle over $X/G$ with fiber $G/H$ and structure group $NH/H$. Therefore, we can state the following:
	



	\begin{thm}[Theorem 5.9 \cite{bredon} page 89] \label{bredon3}
		Let $G$ be a compact Lie group acting continuously on a completely regular Hausdorff space  $X$. Suppose $H$ is a closed subgroup of $G$ (not necessarily normal). If the action of $G$ on $X$ has  only one orbit type $(G/H)$ (all isotropy groups are conjugate to $H$), then the map $\Phi([g,x])= gx$
		$$
		\Phi: G \times_{NH} X^H \rightarrow X
		$$
		is a $G$-homeomorphism. 
	\end{thm}
	
	For a compact Lie group $G$, many $G$-spaces are of the form $G \times_H X$ for $X$ an $H$-space.  Consider a $G$-space $Y$ with a $G$-equivariant continuous map $f: Y \rightarrow G/H$.
	If we take $X=f^{-1}(eH)$, the natural map
	$$
	F : G \times_H X \rightarrow Y
	$$
	is a $G$-equivariant homeomorphism (assuming $G$ is a compact Lie group and $H$ closed).
	
	We aim to generalize this theorem to more diverse groups. This generalization corresponds to conditon (S) of \cite{uribe} and it is satisfied, for example, if the group  is locally compact, second countable and has finite covering dimension. In particular, Lie groups satisfy condition (S).

	\begin{thm}[\cite{uribe} Lemma 4.2] \label{conditionS}Let $ f : Y \rightarrow G / H$  be a $ G$-equivariant  map for some subgroup $H \subseteq G$.  Suppose that condition (S) is satisfied. Then the $G$-equivariant map  $ F: G \times_H f^{-1}(eH) \rightarrow Y$ given by $[g, e] \rightarrow  g \cdot e$  is a $G$-homeomorphism.
	\end{thm}

	We will use the previous theorem to study proper actions of Lie groups with only one isotropy type. When the Lie group acts smoothly and properly on a manifold, the results below follow for example from \cite{Kolk} Theorem 2.6.7, Part iii)  and iv) respectively.
	\begin{thm}\label{Th4.1}
		Let $G$ be a Lie group acting continuously and properly on a manifold $M$. Suppose $H$ is a closed subgroup of $G$ (not necessarily normal). If the action of $G$ on $M$ has  only one orbit type $(G/H)$ (all isotropy groups are conjugate to $H$), then the action of $NH/H$ on $M^H$ is free, and $M^H/NH$ is homeomorphic to $M/G$. Hence, there is a principal $NH/H$-bundle:
		$$
		NH/H \rightarrow M^H \rightarrow M/G
		$$
	\end{thm}
	
	\begin{thm}
		Keeping the same assumptions as in Theorem \ref{Th4.1}, the map 
	\begin{align*}
		\Phi:G \times_{NH/H} M^H &\rightarrow M\\ [gH,x]&\mapsto gx
		\end{align*}
		is a diffeomorphism of  $G$-spaces. 
	\end{thm}
	
	The following theorems extend the previous results to proper actions of Lie groups. The key fact is the existence of tubes and slices for proper actions of groups. In the case of completely regular Hausdorff spaces, see Theorem 2.3.3 of \cite{palais}; for $G$-CW complexes, see Theorem 7.1 of  \cite{uribe}.
	
	\begin{thm} \label{oneisoproper}
		Let $G$ be a Lie group acting properly  on a completely regular Hausdorff space $X$. Suppose $H$ is a compact subgroup of $G$ (not necessarily normal). If the action of $G$ on $X$ has  only one orbit type $(G/H)$ (all isotropy groups are conjugate to $H$), then the map $\Phi([g,x])=gx$
		$$
		\Phi : G \times_{NH} X^H \rightarrow X
		$$
		is a $G$-homeomorphism. 
		\begin{proof}
			Consider the function $\phi : X \rightarrow G/NH$ defined for $x \in X$ as follows: the stabilizer at $x$ is $gHg^{-1}$ for some $g \in G$, and therefore the stabilizer of  $g^{-1}x$ is $H$. $\phi$ is defined sending $x$ to the coset $gNH$.
			
			This function is $G$-equivariant because $\overline{g}x$ has stabilizer $\overline{g}gHg^{-1}\overline{g}^{-1}$ and therefore $\phi(\overline{g}x) = \overline{g}gNH=\overline{g}\phi(x)$.
			
			If $\phi(x) =gNH$, then $G_x = gHg^{-1}$ so that $x \in X^{gHg^{-1}}$. Conversely, if $x \in X^{gHg^{-1}}$, then $gHg^{-1} \subseteq G_x$, and since we are assuming only one isotropy type, then $G_x = \overline{g} H \overline{g}^{-1}$ for some $\overline{g} \in G$. Now we know that
			$$
			gHg^{-1} \subseteq \overline{g} H \overline{g}^{-1} \text { implies } gHg^{-1} =  \overline{g} H \overline{g}^{-1}.
			$$
			Since $H$ is compact, $gH g^{-1}$ and $\overline{g}H \overline{g}^{-1}$ have a finite number of components, and both are Lie groups. Therefore 
			$$
			\phi(x) = gNH  \iff x \in X^{gHg^{-1}}.
			$$
			
			Since the action of $G$ on $X$ is proper and $X$ is a completely regular Hausdorff space, by \cite{palais}, Theorem 2.3.3, there is a slice and tube around $x$. Take a slice at $x$; this is a subset $S \subseteq X$ that is $G_x$-equivariant and such that the natural map $G \times_{G_{x}} S \rightarrow X$ is a $G$-homeomorphism onto an open subset $U$ of $X$. 
			
			Since  the stabilizer at $x$ is $gHg^{-1}$, the tube $U$ has isotropy at $[\overline{g},s]$ given by $G_{[\overline{g},s]} = \overline{g}{({G_{x}})_s}\overline{g}^{-1}  \subseteq \overline{g}gHg^{-1} \overline{g}^{-1}$. But since the isotropy type of $\overline{g}s$ is also $H$, there exists $\widetilde{g}$ such that
			$$
			G_{[\overline{g},s]} = \widetilde{g} H \widetilde{g}^{-1} \subseteq \overline{g}gHg^{-1} \overline{g}^{-1}.
			$$
			
			But since $H$ is compact, $\widetilde{g} H \widetilde{g}^{-1}$ and $\overline{g}gHg^{-1} \overline{g}^{-1}$ have a finite number of components, and since they are both Lie groups, once again we know that, $\widetilde{g} H \widetilde{g}^{-1} \subseteq \overline{g}gHg^{-1} \overline{g}^{-1}$ 
			implies $\widetilde{g} H \widetilde{g}^{-1} = \overline{g}gHg^{-1} \overline{g}^{-1}$, and therefore $\widetilde{g}NH = \overline{g}g  NH$. 
			On the tube $G \times_{G_x} S$ we have 
			$$
			\phi([\overline{g},s]) = \overline{g} g NH
			$$
			which means $\phi$ is constant on the slice and is continuous on the tube at $x$. \
			
			As $\phi^{-1}(eNH) =X^H$, using Theorem \ref{conditionS} the proof is finished.
		\end{proof}
	\end{thm}

	Similarly, for $G$-CW complexes, existance of slices are established by Theorem 7.1 of \cite{uribe}. Therefore the following theorem holds.
	
	\begin{thm} \label{oneisoproperCW}
		Let $G$ be a Lie group acting properly  on a $G$-CW complex $X$. Suppose $H$ is a compact subgroup of $G$ (not necessarily normal). If the action of $G$ on $X$ has  only one orbit type $(G/H)$ (all isotropy groups are conjugate to $H$), then the map $\Phi([g,x])=gx$
		$$
		\Phi : G \times_{NH} X^H \rightarrow X
		$$
		is a $G$-homeomorphism. 
	\end{thm}

	\subsection{K-theory of actions with only one isotropy type}
	
	In this section, we study $G$-equivariant K-theory for actions with only one isotropy type. The are two straightforward types of actions with one isotropy type.
	\begin{enumerate}
		\item  Free actions
		\item Trivial actions
	\end{enumerate}
	
	For free actions of a compact Lie group $G$, we know that the equivariant K-theory is the K-theory of the quotient:
	
	$$
	K^*_G(X) \cong K^*(X/G).
	$$
	
	In the case of trivial actions, the equivariant  K-theory is the tensor product of the group ring with the K-theory of the space, which can be written as a sum over the irreducible representations of $G$ of the K-theory of the space:
	$$
	K^*_G(X) \cong R(G) \otimes K^*(X) \cong \bigoplus_{[\rho] \in   \Irr(G)} K^*(X).
	$$More generally, the $G$-action on $G/H \times X$, acting only on the first component, has one isotropy type, and the equivariant K-theory is:
	$$
	K^*_G(G/H \times X) \cong K^*_H(X) \cong R(H) \otimes K^*(X) \cong \bigoplus_{[\rho] \in  \Irr(H)} K^*(X).
	$$

	We will prove a decomposition formula for actions with only one isotropy type that generalizes the previous examples.
	
	If $G$ is a Lie group satisfying the assumption (K) and acting properly on the space $X$
	with only one orbit type $(G/H)$, where $H$ is a compact subgroup of $G$, we know that there is a $G$-homeomorphism
	$$
	\Phi: G \times_{NH} X^H \rightarrow X.
	$$
	
	Since $H$ is closed, the normalizer $NH$ is a closed subgroup. From the $G$-homeomorphism and induction structure of K-theory, it follows that:
	$$
	K_G^*(X) \cong  K^*_G(G\times_{NH} X^H) \cong K_{NH}^*(X^H).
	$$
	
	$H$ is a closed normal subgroup of $NH$ that acts trivially on $X^H$. From \ref{decompintro} there is a decomposition
	$$
	K_{NH}^*(X^H) \cong  \bigoplus_{[\rho] \in NH \backslash \Irr(H)}
	{}^{\widetilde{W}_{[\rho]}} K^*_{W_{[\rho]}}(X^H),$$
	where the sum is over  representatives of the orbits of the 
	$NH$-action on the set of isomorphism classes of irreducible $H$-representations, and $W_{[\rho]}=NH_{[\rho]}/H$.

	We have an extension of Lie groups
	$$
	1  \to H\stackrel{\iota}\rightarrow NH\stackrel{\pi}\rightarrow NH/H \to  1
	$$
	and for each $\rho : H \to U(V_{\rho})$  complex, finite-dimensional, 
	irreducible representation of $H$, there are central extensions
	\[
	1 \to S^{1} \stackrel{\widetilde{\iota}}{\rightarrow} \widetilde{W}_{[\rho]} \rightarrow W_{[\rho]} \to 1.
	\]
	
	Now, under the conditions of only one isotropy type, $NH/H$ acts freely on $X^H$ and therefore $W_{[\rho]}$ acts freely on $X^H$. Thus
	$$
	{}^{\widetilde{W}_{\rho}} K^*_{W_\rho}(X^H) \cong {}^{\alpha_{[\rho]}} K^*(X^H/W_{[\rho]})
	$$
	where $\alpha_{[\rho]}$ is the twisting coming from the central extension as in Section \ref{proper$K$-theory}. 

	$NH_{[\rho]}$ is a stabilizer of the action of $NH$ on $\Irr(H)$, which is a discrete space. Therefore $NH/NH_{[\rho]}$ is a discrete space. Thus,  there is a covering space: 
	$$
	\xymatrix{
		X^H / NH_{[\rho]} \ar[d] \\
		X^H/NH
	}.
	$$

	But $G \times_{NH} X^H$ is $G$-homeomorphic to $X$ and therefore 
	$$
	X/G \cong X^H / NH
	$$
	and there is a covering space $X^H/NH_{[\rho]} \rightarrow X/G$.

	Since $H$ acts trivially on $X^H$, 
	$$
	X^H/(NH/H) = X^H /NH \cong X/G
	$$
	and $X^H/NH_{[\rho]} = X^H/W_{[\rho]}$. These observations put together result in
	$$
	K_G^*(X) \cong K_{NH}^*(X^H) \cong \bigoplus_{[\rho] \in NH \backslash \Irr(H)}
	{}^{\widetilde{W}_{[\rho]}} K^*_{W_{[\rho]}}(X^H)  \cong  \bigoplus_{[\rho] \in NH \backslash \Irr(H)}   {}^{\alpha_{[\rho]}} K^*(X^H/{W}_{[\rho]})
	$$
	and $X^H/ {W}_{[\rho]} = X^H/NH_{[\rho]} \rightarrow X^H/NH = X/G$ is a covering space with fiber $NH/NH_{[\rho]}$.
	
	All these observations  provide a more general version of Theorem 7 in \cite{wassermann}, which is applicable to actions of compact Lie groups with only one orbit type.
	
	\begin{thm} \label{oneiso}
		Supppose the Lie groups  $G$ and $NH$ satisfy assumption (K) and  act properly  on a locally compact Hausdorff second countable space $X$  with only one orbit type $(G/H)$. Then 
		$$
		K_G^*(X) \cong   \bigoplus_{[\rho] \in NH \backslash \Irr(H)}   {}^{\alpha_{[\rho]}} K^*(X^H/NH_{[\rho]})
		$$
		where for every irreducible representation $\rho$ of $H$ there is a covering space of degree $[NH_{[\rho]}:NH]$
		$$
		X^H/NH_{[\rho]} \rightarrow X/G
		$$
		and a central extension 
		\[
		1 \to S^{1} \stackrel{\widetilde{\iota}}{\rightarrow} \widetilde{W}_{[\rho]} \rightarrow W_{[\rho]} \to 1
		\]
		classified by a cohomology class $\alpha_{\widetilde{W_{[\rho]}}} \in H^2(BW_{[\rho]};S^1)$, and
		$\alpha_{[\rho]} \in H^2(X^H/NH_{[\rho]};S^1)$ is the pullback $p^*(\alpha_{\widetilde{W_{[\rho]}}})$ with 
		$$
		\xymatrix{
			EW_{[\rho]} \times_{W_{[\rho]}} X^H \ar[rr]^{p} \ar[d]^{\simeq} & & BW_{[\rho]} \\
			X^H/NH_{[\rho]} \ar[urr]
		}
		$$
		
	\end{thm}
	
	Theorem \ref{oneiso} generalizes Theorem 3.9 of \cite{GU}.

	\section{Examples}\label{sectionaa}
	In this section, we show some interesting examples using our Theorem \ref{decomp}. The Example \ref{ejemploandres} also uses results from Section \ref{sectionb}.
	\begin{exam}\rm
		Consider the group
		$$
		D_8= \mathbb{Z}/4 \rtimes \mathbb{Z}_2 = \left\langle a , b \mid a ^ { 4 } = b ^ { 2 } = e , b a b ^ { - 1 } = a ^ { - 1 } \right\rangle
		$$
		Let $\rho$ be the irreducible representation of $\mathbb{Z}/4$ given by counterclockwise rotation by $\pi/2$
		Take $A=\mathbb{Z}/4$. The group $D_8$ acts on
		$$
		\operatorname { Irr } ( \mathbb { Z }/4 ) = \{ [1] , [\rho] , [\rho ^ { 2 }] , [\rho ^ { 3 }] \}$$
		with orbits
		$$D _ { 8 } \backslash \operatorname { Irr } ( \mathbb { Z }/4 ) = \{ \{ [1] \} , \left\{ [\rho] , [\rho ^ { 3 }] \right\} , \left\{ [\rho ^ { 2 }] \right\} \}. 
		$$
	For the representations $[1]$ and $[\rho^2]$, we have that $G_{[1]} =G_{[\rho^2]}= D_8$, and therefore $Q_{[1]}=Q_{[\rho^2]} = D_8 /\mathbb{Z}/4 =\mathbb{Z}/2$. For the representation $[\rho]$, we have that $G _ { [ \rho ] } = \mathbb {Z}/4$, and therefore $Q_{ [ \rho ] }= 1$. As $\rho$ is 1-dimensional, Remark \ref{abelian-kernel} implies that twistings are trivial.

		If $X$ is a space with an action of $D_8$ that restricts to the  trivial action of $ \mathbb{Z}/4$, then
		$$
		K_{D_8}^*(X) \cong K_{\mathbb{Z}/2}^* (X) \oplus K^*(X) \oplus K_{\mathbb{Z}/2}^* (X).
		$$
		
		Compare with example 6.1 of \cite{GU}.
		
	\end{exam}

	\begin{exam}\rm\label{ejemploQ8}
		Consider the quaternions $Q_8= \{ \pm 1, \pm i, \pm j, \pm k\}$. It is a central extension of $\mathbb{Z}/2 \times \mathbb{Z}/2$ by $A=\{1,-1\}$:
		\begin{equation*}
		0\to A \to Q_8\to \mathbb{Z}/2 \times \mathbb{Z}/2\to 0.\end{equation*}
		
		Since $A$ is central, $Q_8$ acts trivially on $\operatorname { Irr } ( A ) = \{1,\sign\}$ with $G_{[1]} =G_{[\sign]}= Q_8$ and $Q_{[1]}=Q_{[\sign]} = Q_8 /A =\mathbb{Z}/2 \times \mathbb{Z}/2$.
		
		The trivial representation extends to $Q_8$ then by Prop. \ref{notwist} the twist is trivial; but the $\sign$ representation does not extend to $Q_8 = G_{[\sign]}$. Therefore, again by Prop. \ref{notwist} there is a nontrivial twisting. If $X$ is a space with an action of $Q_8$ with trivial action of $A$, then
		$$
		K_{Q_8}^*(X) \cong K_{\mathbb{Z}/2 \times \mathbb{Z}/2 }^*(X) \oplus {}^{\widetilde{Q}_{[\sign]}} K_{\mathbb{Z}/2 \times \mathbb{Z}/2 }^*(X)
		$$
		and the twisting is represented by the nontrivial element of
		$H^3(\mathbb{Z}/2 \times \mathbb{Z}/2 ; \mathbb{Z}) = \mathbb{Z}/2 $.
		
		Consider the projection $Q_8 \rightarrow \mathbb{Z}/2 \times \mathbb{Z}/2$, and consider the action of $\mathbb{Z}/2 \times \mathbb{Z}/2$ on $\mathbb{S}^2 \times \mathbb{S}^2$. Consider the decomposition 
		$$
		K_{Q_8}^*(\mathbb{S}^2 \times \mathbb{S}^2) \cong K_{\mathbb{Z}/2 \times \mathbb{Z}/2 }^*(\mathbb{S}^2 \times \mathbb{S}^2) \oplus {}^{\widetilde{Q}_{[\sign]}} K_{\mathbb{Z}/2 \times \mathbb{Z}/2 }^*(\mathbb{S}^2 \times \mathbb{S}^2).
		$$
		Since the action is free, 
		$$
		K_{\mathbb{Z}/2 \times \mathbb{Z}/2 }^*(\mathbb{S}^2 \times \mathbb{S}^2) \cong K^*(\mathbb{RP}^2 \times \mathbb{RP}^2 ) 
		$$
		and
		$$
		{}^{\widetilde{Q}_{[\sign]}} K_{\mathbb{Z}/2 \times \mathbb{Z}/2}^*(\mathbb{S}^2 \times \mathbb{S}^2) \cong {}^\alpha K^*(\mathbb{RP}^2 \times \mathbb{RP}^2 ) 
		$$
		where $\alpha$ is the nontrivial class of $H^3(\mathbb{RP}^2 \times \mathbb{RP}^2;\mathbb{Z})=\mathbb{Z}/2$.

		From the Kunneth's theorem for K-theory \cite{AtiyKunneth}, there is a short exact sequence: 
		$$
		0 \rightarrow \bigoplus_{i+j=*} K^i(\mathbb{RP}^{2}) \otimes K^j(\mathbb{RP}^{2}) \rightarrow K^*(\mathbb{RP}^{2} \times \mathbb{RP}^{2}) \rightarrow \operatornamewithlimits{Tor}\limits_{i+j=*+1}(K^i(\mathbb{RP}^{2}), K^j(\mathbb{RP}^{2})) \rightarrow 0.
		$$
		
		For $*=0$, there is no $\Tor$ term and therefore the K-theory is just the tensor product, which is $(\mathbb { Z } \oplus \mathbb{Z}/2) \otimes (\mathbb { Z } \oplus \mathbb{Z}/2) = \mathbb { Z } \oplus (\mathbb{Z}/2)^3 $.
		
		For $*=1$ there is no tensor term and the K-theory is the $\Tor$ part, which is $Tor ((\mathbb { Z } \oplus \mathbb{Z}/2),(\mathbb { Z } \oplus \mathbb{Z}/2)) = \mathbb{Z}/2 $
		
		$$
		K^*(\mathbb{RP}^{2} \times \mathbb{RP}^{2}
		) \cong \left\{ \begin{array} { c l } { \mathbb { Z } \oplus \left ( \mathbb{Z}/2 \right ) ^3 , } & { * = 0 } \\ { \mathbb{Z}/2 , } & { *=1 } \end{array} \right..
		$$

		Let us compute the twisted K-theory ${}^\alpha K^*(\mathbb{RP}^2 \times \mathbb{RP}^2 )$ with the twisted Atiyah-Hirzebruch spectral sequence \cite{AtiySegAHSS}.
		With coefficients $\Z/2$ we have
		$$
		H^*(\mathbb{RP}^2 \times \mathbb{RP}^2;\Z/2) \cong \Z/2[x,y]/(x^3,y^3).
		$$

		On the other hand, for the cohomology  with integer coefficients,
		$$
		H^*(\mathbb{RP}^2 \times \mathbb{RP}^2;\mathbb{Z}) \cong \mathbb Z[x_2,y_2,z_3]/ (2x_2,2y_2,2z_3,x_2^2,y_2^2,z_3^2,x_2z,y_2z_3)
		$$
		with $x_2$ and $y_2$ having degree $2$ and $x_2 \Mod{2} =x^2$, $y_2 \Mod{2} = y^2$,  $z_3$ having degree $3$ and $z_3 \Mod{2} = x^2y+xy^2$. We have the following components in each degree
		\begin{center}
			\begin{tabular}{ |c|c|c|c|c|c|c| }
				\hline
				dim 0 & dim 1 & dim 2 & dim 3 & dim 4 & dim 5 & dim 6  \\ 
				\hline
				$\mathbb{Z}$ &  $0$ & $(\Z/2)^2$ & $\Z/2$  & $\Z/2$ & $0$ & $0$ \\ \hline
			\end{tabular}
		\end{center}
		Then the third page of the Atiyah-Hirzebruch spectral sequence for twisted K-theory looks like
		
		$$
		\begin{array}{c|ccccccc}
		\vdots & \vdots & \vdots & \vdots & \vdots & \vdots & \vdots & \vdots \\
		2 &  \mathbb{Z} &  0 & \Z/2 \oplus \Z/2   & \Z/2 & \Z/2 & 0 & 0  \\
		1 & 0 & 0 & 0 & 0 & 0 & 0 & 0 \\
		0 & \mathbb{Z} &  0 & \Z/2 \oplus \Z/2   & \Z/2 & \Z/2 & 0 & 0\\
		-1 & 0 & 0 & 0 & 0 & 0 & 0 & 0 \\
		-2 & \mathbb{Z} &  0 & \Z/2 \oplus \Z/2   & \Z/2 & \Z/2 & 0 & 0\\
		\vdots & \vdots & \vdots & \vdots & \vdots & \vdots & \vdots & \vdots \\
		\hline & 0 & 1 & 2 & 3 & 4 & 5 & 6 
		\end{array}
		$$
		where
		$d^3 :  H^0 \to H^3$ sends 1 to $\alpha$
		and		$d^3 : H^1 \to H^4$ is zero. 
		 The $E_4$-page is the $E_\infty$-page and gives
		$$
		E_\infty^{0,2k} = 2 \mathbb{Z}, E_\infty^{2,2k} = \Z/2\oplus \Z/2, E_\infty^{3,2k} = 0, E_\infty^{4,2k} = \Z/2
		$$ 
		Let us see what the result is for ${}^\alpha K^*(X)$. Working backwards, $F^4 ({}^\alpha K^{0}(\mathbb{RP}^2 \times \mathbb{RP}^2)) = \Z/2$ which is also $F^3$, and then $F^2 / F^3 \cong \Z/2 \oplus \Z/2$, so we need to find possible extensions
		$$
		\{0\}\rightarrow \Z/2 \rightarrow F^2 ({}^\alpha K^{0}(\mathbb{RP}^2 \times \mathbb{RP}^2)) \rightarrow \Z/2 \oplus \Z/2  \rightarrow \{0\}
		$$
		which could be $\Z/4 \oplus \Z/2$ or $\Z/2 \oplus \Z/4$ or $(\Z/2)^3$.

		So $F^2=F^1$, and thus we have the reduced K-theory equal to $(\Z/2)^3$ or $\Z/2 \oplus \Z/4$
		
		Also note that in the extension
		$$
		\{0\}\rightarrow F^1 \rightarrow F^0 ({}^\alpha K^{0}(\mathbb{RP}^2 \times \mathbb{RP}^2)) \rightarrow \mathbb{Z}  \rightarrow \{0\}
		$$
		since $\mathbb{Z}$ es free, the extension is trivial. Therefore,
		
		$$
		{}^\alpha K^0(\mathbb{RP}^2 \times \mathbb{RP}^2) \cong \mathbb{Z} \oplus (\Z/2)^3
		$$
		or
		$$
		{}^\alpha K^0(\mathbb{RP}^2 \times \mathbb{RP}^2) \cong \mathbb{Z} \oplus \Z/4 \oplus \Z/2.
		$$
		
		For $*=1$, the only $E_\infty$-page that contributes is $E_\infty^{3,-2}=0$
		which gives
		$$
		{}^\alpha K^1(\mathbb{RP}^2 \times \mathbb{RP}^2) = 0.
		$$
		
		To be able to resolve the extension problems, we will use twisted K-theory with $\mathbb{Z}/2$-coefficients denoted by ${}^{\alpha}K^*(X;\Z/2)$. 
		
		
		 Twisted K-theory  $ { }^\alpha K^n ( X)$ can be realized as the K-theory of the C*-algebra $A_{\alpha}$ of sections of certain bundle of compact operators associated to $\alpha$. For details see Prop. 2.1 in \cite{rosenberg-ct}. Along this lines, if $C_2$ is the C*-algebra defined in \cite[Eqn. (1.1)]{schochet-IV}. We have an isomorphism
  $${}^\alpha K^n \left( X ;\Z/2 \right)\cong K_n(A_{\alpha}\otimes C_2).$$
  The universal coefficent theorem for K-theory of C*-algebras (\cite{rosenberg-schochet}) gives us the following sequence for $n=0$.
		\begin{align*} 0 \rightarrow {}^ { \alpha }K ^{ 0 }  ( \mathbb{RP}^2 \times \mathbb{RP}^2) \otimes \Z/2 &\rightarrow {}^\alpha K^0 \left(\mathbb{RP}^2\times\mathbb{RP}^2;\Z/2 \right) \rightarrow \\&\operatorname { Tor } _ { 1 } \left( {}^ { \alpha}K ^ {   1 }(\mathbb{RP}^2 \times \mathbb{RP}^2 )    , \Z/2 \right)\to0.\end{align*}
		Since 
		$$
		{}^\alpha K^1(\mathbb{RP}^2 \times \mathbb{RP}^2) = 0
		$$
		this gives $ {}^ { \alpha }K ^{ 0 }  ( \mathbb{RP}^2 \times \mathbb{RP}^2) \otimes \Z/2 \cong  {}^\alpha K^0 \left( \mathbb{RP}^2\times\mathbb{RP}^2 ;\Z/2 \right)$.
		
		Let us compute ${}^\alpha K^0 \left( \mathbb{RP}^2\times\mathbb{RP}^2 ;\Z/2 \right)$  using the twisted Atiyah Hirzebruch spectral sequence. 
		
		$$
		H^*(\mathbb{RP}^2 \times \mathbb{RP}^2;\Z/2) \cong \Z/2[x,y]/(x^3,y^3).
		$$
		On each degree we have
		
		\begin{center}
			\begin{tabular}{ |c|c|c|c|c|c|c| }
				\hline
				dim 0 & dim 1 & dim 2 & dim 3 & dim 4  \\ 
				\hline
				$\Z/2$ &  $(\Z/2)^2$ & $(\Z/2)^3$ & $(\Z/2)^2 $ & $\Z/2$  \\ \hline
				1 & $x,y$& $x^2,xy,y^2$ & $x^2y,y^2x$ & $x^2y^2$ \\
				\hline
			\end{tabular}
		\end{center}
		And the third page of the spectral sequence looks like
		
		$$
		\begin{array}{c|ccccccc}
		\vdots & \vdots & \vdots & \vdots & \vdots & \vdots & \vdots & \vdots \\
		2 &  \Z/2 &  (\Z/2)^2 & (\Z/2)^3  & (\Z/2)^2 & \Z/2 & 0 & 0  \\
		1 & 0 & 0 & 0 & 0 & 0 & 0 & 0 \\
		0 &  \Z/2 &  (\Z/2)^2 & (\Z/2)^3  & (\Z/2)^2 & \Z/2 & 0 & 0 \\
		-1 & 0 & 0 & 0 & 0 & 0 & 0 & 0 \\
		-2 &  \Z/2 &  (\Z/2)^2 & (\Z/2)^3  & (\Z/2)^2 & \Z/2 & 0 & 0 \\
		\vdots & \vdots & \vdots & \vdots & \vdots & \vdots & \vdots & \vdots \\
		\hline & 0 & 1 & 2 & 3 & 4 & 5 & 6 
		\end{array}
		$$
  
        Let us denote by  $Q_i$ the Milnor primitives at prime $2$, defined inductively by $Q_0=Sq^1$ and $ Q _ { j + 1 } = S q ^ { 2 ^ { j+1 } } Q _ { j } - Q _ { j } S q ^ { 2 ^ { j+1 } } $. Then the differentials are given by
		\[ d _ { 2 ^ { n + 1 } - 1 } \left( x v _ { n } ^ { k } \right) = \left( Q _ { n } ( x ) + ( - 1 ) ^ { | x | } x \cup \left( Q _ { n - 1 } \cdots Q _ { 1 } ( H ) \right) \right) v _ { n } ^ { k - 1 } \]
		(from \cite{sati}) and for $n=1$,
		$$ d _ { 3 } \left( x v _ { 1 } ^ { k } \right) = \left( Q _ { 1} ( x ) +  x \cup   H \Mod{2}  \right) v _ { 1 } ^ { k - 1 }
		$$
		
		Therefore $Q_1= Sq^2Sq^1-Sq^1Sq^2$. On the other hand, $Sq^1Sq^2 = Sq^3$ (see for example Mosher-Tangora \cite{mosher} page 23), making $Q_1 = Sq^2Sq^1+Sq^3$.\\
		Thus we conclude:
		$$
		d^3(x) = Q_1(x) + x \cup  H \Mod{2}.
		$$
		In our example, $Q_1(1)=0$ and\\
		$$Q_1(x)=Sq^2Sq^1(x)=Sq^2(x^2)=x^4=0,$$ $$Q_1(y)=Sq^2Sq^1(y)=Sq^2(y^2)=y^4=0,$$
		therefore $d^3 :  H^0 \to H^3$ sends 1 to $\alpha \Mod{2} = x^2y+yx^2$;
		$d^3 : H^1 \to H^4$ sends $x \to x (x^2y+y^2x) = y^2x^2 $ and $y \to y (x^2y+y^2x) = y^2x^2 $; thus $x+y \in \ker(d^3)$.

		The $E_4$-page is the $E_\infty$-page and gives
		$$
		\begin{array}{c|ccccccc}
		\vdots & \vdots & \vdots & \vdots & \vdots & \vdots & \vdots & \vdots \\
		2 &  0 &  \Z/2 & (\Z/2)^3  & \Z/2 & 0& 0 & 0  \\
		1 & 0 & 0 & 0 & 0 & 0 & 0 & 0 \\
		0 &  0 &  \Z/2 & (\Z/2)^3  & \Z/2 & 0 & 0 & 0 \\
		-1 & 0 & 0 & 0 & 0 & 0 & 0 & 0 \\
		-2 &  0 &  \Z/2 & (\Z/2)^3  & \Z/2 & 0 & 0 & 0 \\
		\vdots & \vdots & \vdots & \vdots & \vdots & \vdots & \vdots & \vdots \\
		\hline & 0 & 1 & 2 & 3 & 4 & 5 & 6 
		\end{array}
		$$

		$$
		E_\infty^{1,2k} = \Z/2, E_\infty^{2,2k} =  (\Z/2)^3, E_\infty^{3,2k} = \Z/2.
		$$ 
		In this manner we obtain:
		$$
		{}^\alpha K^{ 0 } \left( \mathbb{RP}^2 \times \mathbb{RP}^2 ; \Z/2\right ) \cong (\Z/2)^3
		$$
		and
		$$
		{}^\alpha K^{ 1 } \left( \mathbb{RP}^2 \times \mathbb{RP}^2 ; \Z/2\right ) \cong (\Z/2)^2.
		$$
		We have either
		$$
		{}^\alpha K^0(\mathbb{RP}^2 \times \mathbb{RP}^2) \cong \mathbb{Z} \oplus (\Z/2)^3
		$$
		or
		$$
		{}^\alpha K^0(\mathbb{RP}^2 \times \mathbb{RP}^2) \cong \mathbb{Z} \oplus \Z/4 \oplus \Z/2
		$$
		but tensoring with $\Z/2$ we see that the only possibility consistent with the  universal coefficient theorem, is
		$$
		{}^\alpha K^0(\mathbb{RP}^2 \times \mathbb{RP}^2) \cong \mathbb{Z} \oplus \Z/4 \oplus \Z/2.
		$$
		Summarizing, 
		$$
		{}^\alpha K^0(\mathbb{RP}^2 \times \mathbb{RP}^2) \cong \left\{ \begin{array} { c l } { \mathbb{Z} \oplus \Z/4 \oplus \Z/2,} & { * = 0 } \\ { 0 , } & { *=1 } \end{array} \right.,
		$$
		and all of these observations put together result in
		$$
		K^*_{Q_8}(\mathbb{S}^{2} \times \mathbb{S}^{2}
		) \cong \left\{ \begin{array} { c l } { \left (\mathbb { Z } \right)^2 \oplus \left ( \mathbb{Z}/2 \right ) ^4 \oplus \mathbb{Z}/4 , } & { * = 0 } \\ { \mathbb{Z}/2 , } & { *=1 } \end{array} \right..
		$$
		Compare with example in page 22 in \cite{wassermann} and  with example 6.2 of \cite{AGU}, where $K^*_{D_8}(\mathbb{S}^{\infty} \times \mathbb{S}^{\infty})$ is computed with the help of the Atiyah-Segal completion theorem.  
		
	\end{exam}

	\begin{exam}\rm
	    Consider the Lie group $SU(2)$. The \emph{spinor map} $\Spin:SU(2)\rightarrow SO(3)$ whose value at $u\in SU(2)$ is the linear transformation $u\vec{x}u^*$, where $\vec{x}\in \mathfrak{su}(2)\cong \R^3$. On the other side, the group $SO(3)$ acts on itself via the action defined as $s\cdot a=sas^{-1}$, for all $s,a\in SO(3)$. By composing these two maps, we define an action of $SU(2)$ over $SO(3)$ as $u\cdot a=\Spin(u)a\Spin(u)^{-1}$ for each $u\in SU(2)$ and $a\in SO(3)$. Due to 
	    $$\Spin(u)=\begin{pmatrix} 1&0&0\\
	    0&1&0\\
	    0&0&1\end{pmatrix} \mbox{  if and only if  }
	    u=\pm \begin{pmatrix} 1&0\\
	    0&1 \end{pmatrix}$$the kernel of $\Spin$ can be identified with the group $\Z/2$. Applying Theorem \ref{decomp} with $G=SU(2)$, $A=\{Id, -Id\}$ and $X=SO(3)$, we observe that, in the same way as the Example \ref{ejemploQ8}, $SU(2)$ acts trivially on $\operatorname { Irr } ( \Z/2 ) = \{1,\sign\}$ with $G_{[1]} =G_{[\sign]}= SU(2)$ and $Q_{[1]}=Q_{[\sign]} = SU(2) /(\mathbb{Z}/2) =SO(3)$. While the trivial representation extends to $SU(2)$, the $\sign$ representation does not extend to $SU(2)$. Otherwise, we should have a non-trivial homomorphism $s$ from $SU(2)$ to $\Z/2$. That means we should obtain a split central extension:
	    $$1\to\ker(s)\rightarrow SU(2)\xrightarrow{s} \Z/2\rightarrow 1.$$
     But this is not possible because $\ker(s)$ should have to be a subgroup of $SU(2)$ of the same dimension, so we should get $SU(2)\cong \Z/2\times SU(2)$. This is a contradiction.
  Therefore, there is a nontrivial twisting $\widetilde{Q}_{[\sign]}$ represented by the nontrivial element of
		$H^3(SO(3) ; \mathbb{Z}) = \mathbb{Z}/2 $. Then, by Theorem \ref{decomp}:
	    $$K_{SU(2)}^*(SO(3))\cong K_{SO(3)}^*(SO(3))\oplus {}^{\widetilde{Q}_{[\sign]}}K_{SO(3)}^*(SO(3)),$$ where the induced action from $SO(3)$ on itself is given by conjugation. The detailed computations for the right-hand side of the equation can be found in \cite{FHT}, Section 9.
	\end{exam}

	\begin{exam}\rm
		Let $\St_n(R)$ be the Steinberg group \cite{steinberg}  associated to a commutative ring $R$, and consider the central extension
		\begin{equation}\label{ext}0\to\Z/2\to\St_n(\Z)\to\Sl_n(\Z)\to 0.\end{equation}
		Note that if we take a finite $\Sl_n(\Z)$-CW-complex model of $\underbar{E}\Sl_n(\Z)$, this space with the induced action of $\St_n(\Z)$ is a finite $\St_n(\Z)$-CW-complex model for $\underbar{E}\St_n(\Z)$. 
		
		Moreover, as the extension is central, the action of $\,\Sl_n(\Z)$ over $\Irr(\Z/2)=\{[1],[\sign]\}$ is trivial. and by Remark \ref{abelian-kernel} the $S^1$-central extensions corresponding to each irreducible representation area trivial. On the other hand in \cite{BaVe2016}, calculations of  equivariant (twisted)  K-theory groups for $\underbar{E}\Sl_3(\IZ)$ were obtained.  Then using Theorem \ref{decomp}, we have
		\begin{align*}K_{\St_3(\IZ)}^*(\underbar{E}\St_3(\IZ))\cong & K_{\Sl_3(\IZ)}^*(\underbar{E}\Sl_3(\IZ))\oplus K_{\Sl_3(\IZ)}^*(\underbar{E}\Sl_3(\IZ))\\&\cong \begin{cases}
		\IZ^{16}&\text{ if $*$ is even}\\
		0 &\text{ if $*$ is odd.}
		\end{cases}\end{align*}
	\end{exam}
	\begin{exam}\rm\label{ejemploandres}
		Let us look at actions of $SU(2)$ on simply connected $5$-dimensional manifolds with only one isotropy type and orbit space $S^2$. These manifolds have been classified up to equivariant diffeomorphism in  \cite{Simas}. There can only be three orbit types: $(e)$, $(SU(2)/\mathbb{Z}_2)$ and $(SU(2)/\mathbb{Z}_m)$. We will describe now the possible equivariant diffeomorphic types:

		Consider for natural numbers $l,m$ with $(l,m)=1$,
		$$
		\mathcal { N } _ { m , m } ^ { l } = \mathrm { SU } ( 2 ) \times _ {{ S } ^ { 1 }}   S  ^ { 3 } 
		$$
		with $S^1$ acting on $SU(2) \times S^3$ by 
		$$
		z \cdot (g,(w_1,w_2)) = (g z^{-l}, (z^m w_1, z^m w_2))
		$$
		and $\mathrm{SU}(2)$ acting on $\mathcal { N } _ { m , m } ^ { l }$ by
		$$
		\overline{g} \cdot [g, (w_1,w_2)] = [\overline{g}g,(w_1,w_2)]
		$$
		
		\begin{itemize}
			\item $\mathcal { N } _ { 1 , 1 } ^ { 0}=SU(2) \times S^2$ is the free $SU(2)$-manifold. In this case $$K_{SU(2)}^*(\mathcal{N}_{1,1}^0)\cong K^*(S^2)\cong \mathbb{Z}[H]/(H-1)^2$$ concentrated in degree 0.
			\item  $SO(3) \times S^2$ and $\mathcal { N } _ { 2 , 2 } ^ { 1}$ are the $SU(2)$-manifolds with isotropy  type $(SU(2)/\Z/2)$. $SO(3) \times S^2$ is not simply connected, but $\mathcal { N } _ { 2 , 2 } ^ { 1}$ is in fact simply connected.
		\end{itemize}
		
		For  isotropy type $(SU(2)/\mathbb{Z}_m)$, with $m\geq 3$, we have the manifolds
		$$
		\mathcal { N } _ { m , m } ^ { l }.
		$$ 
		
		Actually, by the theorem of Barden-Smale about  $5$-dimensional simply connected manifolds in \cite{Simas}, Proposition 3 it is proved that:
		$$
		\mathcal { N } _ { m , m } ^ { l } \cong S^3 \times S^2
		$$
 and in Lemma  10 that
 $$
 \mathcal{ N } _ { m , m } ^ { l }/SU(2) \cong S^2.
 $$
		To use theorem \ref{oneiso} we need to know the normalizers and the Weyl groups, from \cite{Simas} Table 1, we have:

		$$
		\begin{array}{l|lll} 
		H & 1 & \Z/2 & \mathbb{Z}/m \\
		NH & SU(2) & SU(2) & Pin(2)  \\
		NH/H & SU(2) & SO(3) & Pin(2) 
		\end{array}
		$$
		
		

		%

		Let us consider the case $m \geq 3$. The normalizer in this case is $N\mathbb{Z}/m=Pin(2)$. 
		By theorem \ref{oneiso}
		
		$$
		K_{SU(2)}^*(\mathcal { N } _ { m , m } ^ { l }  ) \cong   \bigoplus_{[\rho] \in Pin(2) \backslash \Irr(\mathbb{Z}/m)}   {}^{\alpha_{[\rho]}} K^*(\left (\mathcal { N } _ { m , m } ^ { l } \right )^{\mathbb{Z}/m}/Pin(2)_{[\rho]})
		$$
		where for every irreducible representation $\rho$ of $\mathbb{Z}/m$ there is a covering space  
		$$
		\left ( \mathcal { N } _ { m , m } ^ { l }\right ) ^{\mathbb{Z}/m}/Pin(2)_{[\rho]} \rightarrow \mathcal { N } _ { m , m } ^ { l }/SU(2)
		$$
		of degree $[Pin(2)_{[\rho]}:Pin(2)]$. But $\mathcal{ N } _ { m , m } ^ { l }/SU(2) \cong S^2$, which is simply connected, and therefore the nontrivial coverings are disconnected,
		$$
		\left ( \mathcal { N } _ { m , m } ^ { l }\right ) ^{\mathbb{Z}/m}/Pin(2)_{[\rho]} \cong S^2 \times [Pin(2)_{[\rho]}:Pin(2)].
		$$
		Since $S^2$ is $2$-dimensional, there are no twistings.

		Therefore
		$$
		K_{SU(2)}^*(\mathcal { N } _ { m , m } ^ { l }  )  \cong \bigoplus_{[\rho] \in Pin(2) \backslash \Irr(\mathbb{Z}/m)}    K^*(S^2 \times [Pin(2)_{[\rho]}:Pin(2)])   
		$$
		which is just
		$$
		\bigoplus_{[\rho] \in Pin(2) \backslash \Irr(\mathbb{Z}_m)} \bigoplus_{[Pin(2)_{[\rho]}:Pin(2)]}  K^*(S^2) \cong  \bigoplus_{[\rho] \in  \Irr(\mathbb{Z}/m)}  K^*(S^2)
		$$
		This can be written as
		
		$$
		K_{SU(2)}^*(\mathcal { N } _ { m , m } ^ { l }  )  \cong R(\mathbb{Z}/m) \otimes K^*(S^2) \cong \mathbb{Z}[\sigma ] / (\sigma^m-1) \otimes \mathbb{Z}[\gamma]/(\gamma-1)^2
		$$
		concentrated in degree zero (the isomorphisms are just additive).  The vector bundle $\gamma$ denotes the tautological 
		line bundle over $S^2$. 
		
		The case of isotropy $\Z/2$ is completely analogue, resulting in a similar answer. We finish by saying that for a simply connected  5-manifold $M$, with action of $SU(2)$ with one isotropy type $(SU(2)/H)$ and orbit space $S^2$, there is an isomorphism of abelian groups
		$$K_{SU(2)}^*(M)\cong R(H)\otimes \Z[\gamma]/(\gamma-1)^2.$$
		
		Note that this represents the equivariant K-theory of  $SU(2)/H \times S^2$, a $5$-dimensional $SU(2)$-manifold with only one isotropy type, yet is not simply connected.
		
	\end{exam}
	
	\appendix
	\section{Additional proofs regarding the assumption (K)}
	We already said (see remark \ref{nota1}) that compact and discrete groups satisfy the collection of axioms we called assumption (K) when one takes care of the class of topological spaces where they act on. In this section we would discuss when the assumption (K) is still fulfilled for a large class of Lie groups. At this very instant, our goal is to prove that the classes of almost connected and linear Lie groups also satisfy the assumption (K). In order to do it, we will use the following equivalent assumption for a Lie group $G$:
	\begin{equation}
	\tag{AN}\label{eq:AN}
	\parbox{\dimexpr\linewidth-4em}{%
		\strut%
		\emph{
			For every finite-dimensional $G$-bundle $E \rightarrow X$ over a $G$-space $X$ such that $X/G$ is a compact space, there exists a finite-dimensional  $G$-representation $V$ and a finite-dimensional $G$-bundle $F \rightarrow X$ such that there is an isomorphism of $G$-vector bundles: $E \oplus F \cong X \times  V$.}  
		\strut
	}\end{equation}

	It was proved in \cite{Phil2}, Lemma 2.2, that assumption (AN) implies that the equivariant K-theory $K_G(X)$ is a $G$-cohomology theory. It is a very well-known fact that for torsion twistings, the associated twisted K-theory is a subring of a nontwisted K-theory represented by a central extension $\tilde{G}$ of $G$ by $S^1$. Then if this central extension satisfies assumtion (AN), it also satisfies assumption (K) as defined in Section \ref{proper$K$-theory}. Also, assumption (AN) is closely related to the property of being \emph{Bredon-compatible}, defined by J. Cantarero in \cite{canta} and \cite{cantat}  for representable groupoids. We use assumpion (AN) here because it happens to be easier to verify it for some particular cases. 
	\begin{thm}
		Let $$1 \to A\to G\xrightarrow{\pi}Q\to 1$$ be an extension of Lie groups, where $A$ is compact. If $G$ satisfies $(AN)$ then $Q$ also does.
		\begin{proof}
			Let $E\to X$ be a $Q$-vector bundle over a $Q$-space $X$. $X$ is a $G$-space via 
			the map $\pi$. Note that this action restricted to the subgroup $A$ is trivial. Similarly, the vector bundle $E$ can be seen as a $G$-vector bundle. As $G$ satisfies  assumption (AN), for the vector bundle $E$ regarded as a $G$-vector bundle there exists a finite-dimensional  $G$-representation $V$ and a finite-dimensional $G$-vector bundle $F \rightarrow X$ such that there is an isomorphism of $G$-vector bundles $E \oplus F \cong X \times  V$. As the group $A$ acts trivially on the space $X$, the image of $E$ (regarded as a $G$-vector bundle) via the natural morphism  $\phi: K_G(X)\to K_{Q}(X/A)=K_Q(X)$ is the vector bundle $E$ itself. Since the application $\phi$ is a ring homomorphism, $$\phi(E \oplus F)=E \oplus \phi(F)=X\times V$$as
			$Q$-vector bundles. This proves that assumption (AN) holds for $Q$. 
		\end{proof}
	\end{thm}
	
	\begin{thm}
		Suppose  $H$ is a closed  subgroup of the Lie group $G$. If $G$ satisfies assumption (AN), then $H$ also does.
		\begin{proof}
			Let $E \rightarrow X$ be a $H$-vector bundle over a $H$-space $X$, $H$ acting on the right on $G$ by multiplication. Consider the $G$-vector bundle
			$$
			G \times_H E \rightarrow G \times_H X 
			$$over the $G$-space $G \times_H X$. Using the assumption (AN), for this vector bundle there exists a $G$-vector bundle $F \rightarrow G \times_H X$ and a $G$-representation $V$ such that there is an isomorphism of $G$-vector bundles over $G \times_H X$
			$$
			(G \times_H E) \oplus F \cong (G \times_H X) \times V.$$
			
			We consider the $H$-equivariant map
			$$
			\varepsilon_X : X \rightarrow G \times_H X,
			$$defined by $x\mapsto [e,x]$. 
			Pulling back with $\varepsilon_X$ there is an isomorphism of $H$-vector bundles over $X$. Therefore
			$$
			\varepsilon_X^*(G \times_H E)\oplus \varepsilon_x^*(F) \cong E\oplus \varepsilon_x^*(F)\cong X\times V|_H
			$$as $H$-equivariant bundles. Then the assumption (AN) holds for $H$.
		\end{proof}
	\end{thm}
	\begin{defn}
		A linear (real) group is a closed subgroup of the group $GL_n(\R)$.
	\end{defn}
	\begin{thm}[\cite{Phil2}, Lemma 2.1]
		If $G$ is a linear group, then it satisfies assumption (AN).
	\end{thm}
	\begin{thm}
		Consider an extension
  $$1\to A\to G\to Q$$for a linear group $G$ with $A$ compact. Let $\rho$ be an irreducible representation of $A$. Consider the central extension $$1 \to S^1\to \widetilde{Q}_{[\rho]}\to Q_{[\rho]} \to 1.$$ Then $\widetilde{Q}_{[\rho]}$ satisfies assumption (AN).
		\begin{proof}
			The group $G_{[\rho]}$ is a subgroup of a linear group, hence it is linear. The group $\widetilde{G}_{[\rho]}$ by definition a subgroup of $G_{[\rho]} \times U(V_{[\rho]})$ which is the product of linear groups. Also, since $\widetilde{G}_{[\rho]}$ is linear, it satisfies (AN) and therefore $\widetilde{Q}_{[\rho]} = \widetilde{G}_{[\rho]}/A$.
		\end{proof}
	\end{thm}
	
	As a result by Phillips, the equivariant K-theory defined with finite-dimensional $\tilde{Q}_{[\rho]}$-vector bundles is a $\tilde{Q}_{[\rho]}$-cohomology theory on the category of locally compact Hausdorff second-countable spaces. Decomposing by the characters of $S^1$, the ${}^{\widetilde{Q}_{[\rho]}}K^*_{Q_{[\rho]}}$-theory is a $Q_{[\rho]}$-equivariant cohomology theory. 

	We recall that a topological group $G$ is called almost connected if $G/G_0$ is compact, where $G_0$ denotes the identity component of $G$. Any compact group is almost connected. Any connected group is almost connected. A Lie group is almost connected precisely when it has finitely many components.
	By Theorem 3.1 in \cite{Phil2}, for every almost connected group $G$ there is a maximal compact subgroup $K$.
	\begin{thm}[Theorem 3.2 in \cite{Phil2}]\label{teophil}
		Let $G$ be an almost connected group, and let $X$ be a proper $G$-space. There is a maximal compact subgroup $K$, and a closed $K$-invariant subset
		$S$ of $X$ such that the natural map $(g,x)\mapsto gx$ is a $G$-homeomorphism from $G\times_KS$ to $X$.
	\end{thm}

	\begin{thm}
		If $G$ is an almost connected Lie group acting on a compact $G$-space $X$, then it satisfies assumption (AN). \begin{proof}
		Let $S$ and $K$ be as in the Theorem \ref{teophil}, thus $X\cong G\times_K S$. Let $E'$ be a finite dimensional $K$-vector bundle over $S$. The vector bundle $G\times_KE'$, is a $G$-vector bundle over the space $G\times_KS$. On the other side, we build on the arguments of example $iii)$ page 132 of \cite{segal}. Given a $G$-vector bundle $E$ over the space $G\times_KS$, we take the homeomorphism $\beta:G\times E\rightarrow G\times E$, defined by $(g,\xi)\mapsto (g,g^{-1}\xi)$. We note that the space $S$ can be identified with the subspace $K\times_KS\subset G\times_KS$, thus we can consider the vector bundle $E'=E|_{K\times_KS}$. We have that $$\beta^{-1}(G\times E')=\{(g,\xi)\in G\times E:g^{-1}\xi\in E'\}.$$ Therefore, restricting to the action of $K$ we obtain a map $\overline{\beta}:\beta^{-1}(G\times E')\rightarrow G\times_KE'$. Finally, we consider the natural open map $\pi:\beta^{-1}(G\times E')\rightarrow E$ defined by $(g,\xi)\mapsto \xi$ in order to induce a homeomorphism: $\phi:E\rightarrow G\times_KE'$, defined as $\phi(\xi)=\overline{\beta}(g,\xi)$ if $\pi(g,\xi)=\xi$. 
  
   As the group $K$ is compact,  it is also a closed subgroup of $G$. Then, by Cartan's closed subgroup theorem (Theorem 20.10 in \cite{lee2003introduction}), $K$ is a compact Lie group and therefore $K$ satisfies (AN). So there is a finite-dimensional $K$-vector bundle $F'$ over $S$ and a finite-dimensional $K$-representation $V$ such that $$E'\oplus F'\cong S\times V.$$ It implies
			$$E\oplus (G\times_KF')\cong X\times(G\times_KV).$$ This proves assumption (AN) for $G$.
		\end{proof}
	\end{thm}
	
	\begin{thm}\label{CocienteAC}
		Suppose  $A$ is a closed, normal subgroup of an almost connected group $G$. Then $Q=G/A$ is almost connected.
		\begin{proof} This theorem is a direct consequence of \cite{prolie}, Lemma 5.7. Nevertheless we provide a sketch of a proof using different arguments. Let $Q_0$ be the identity component of $Q$. Note that $\overline{G_0A}/A\subset Q_0$. Therefore $$Q/Q_0\cong [Q/(\overline{G_0A}/A)]/[Q_0/(\overline{G_0A}/A)].$$
   On the other side, we note that$$Q/(\overline{G_0A}/A)=(G/A)/(\overline{G_0A}/A)\cong G/\overline{G_0A}\cong (G/G_0)/(\overline{G_0A}/G_0).$$ Since $G/G_0$ is a compact space, the quotient space $(G/G_0)/(\overline{G_0A}/G_0)$ also is. So $Q/(\overline{G_0A}/A)$ is also compact and then $Q/Q_0$ is a compact space.
		\end{proof}
	\end{thm}
	
	\begin{thm}\label{indicefinitoAC}
		Suppose  $H$ is a  finite index, closed subgroup of an almost connected Lie group $G$. Then $H$ is almost connected.
		\begin{proof}
			By the finite index of $H$ in $G$, $$G=g_1H\cup g_2H\cup\cdots \cup g_kH,$$for $g_1=e,g_2,\dots ,g_k$ a set of representatives of the quotient $G/H$. Note that for each $i=1,...,k$ the set $g_iH$ is a closed set due it is the image of $H$ via the homeomorphism $g_i:G\rightarrow G$ the which maps $h\mapsto g_ih$. Observe that $G_0$ is contained in $H$, because $$G_0=(G_0\cap g_1H)\cup  (G_0\cap g_2H)\cup\cdots \cup (G_0\cap g_kH)$$then $G_0$ is a disjoint union of finite closed sets in $G_0$, as $G_0$ is connected, $G_0$ is contained in the coset $H$. Then $G_0=H_0$ and $H/H_0=H/G_0$ that is a closed subset of $G/G_0$, as $G/G_0$ is compact then $H/H_0$ is compact. 
		\end{proof}
	\end{thm}
	
	
	\begin{thm}\label{extensionAC}
		Suppose  $1 \to A\to G\to Q\to 1$, where $A$ is a compact normal subgroup of the Lie group $G$, then if $Q$ is almost connected, $G$ is almost connected.
		\begin{proof}
Firstly, we point out that an almost connected Lie group $Q$ only has a finite number of connected components. 

So, let us suppose that $A$ is a compact normal subgroup of $G$. As the right action of $A$ on $G$ is always free and $A$ is a closed Lie subgroup of $G$, the quotient map $\pi:G\rightarrow Q=G/A$ is a principal $A$-bundle. Then, by using the long exact sequence of homotopy groups associated to the fibration $A\rightarrow G\rightarrow Q$, we have a long exact sequence: \begin{align}\label{sequence}\cdots\rightarrow \pi_0(A)\rightarrow \pi_0(G)\rightarrow \pi_0(Q)\to0.\end{align} Due to the compactness of $A$, the group $\pi_0(A)\cong A/A_0$ is finite. By the hypothesis $Q$ is almost connected, then $\pi_0(Q)\cong Q/Q_0$ is finite because $Q$ is a Lie group. As the sequence (\ref{sequence})  is exact, we have that $\pi_0(G)=G/G_0$ must be a finite group and therefore, a compact group. 
			
		\end{proof}
	\end{thm}

	\begin{thm}Consider an extension
  $$1\to A\to G\to Q$$for an almost connected group $G$ with $A$ compact. Let $\rho$ be an irreducible representation of $A$. Consider the associated central extension
		 $$1 \to S^1\to \widetilde{Q}_{[\rho]}\to Q_{[\rho]} \to 1.$$ Then $\widetilde{Q}_{[\rho]}$ is almost connected.
		\begin{proof}
			$G_{[\rho]}$ is a subgroup  of finite index within an almost connected subgroup; therefore, by Theorem \ref{indicefinitoAC}, it is almost connected.  Furthermore, by Theorem \ref{CocienteAC}, the quotient of an almost connected group by a normal subgroup is almost connected, thus $Q_{[\rho]}$ is an almost connected group.
			
			Since $\widetilde{Q}_{[\rho]}$ is an extension of $Q_{[\rho]}$ by $S^1$ which is compact, by Theorem \ref{extensionAC}, $\widetilde{Q}_{[\rho]}$ is almost connected.
		\end{proof}
	\end{thm}
	

	\bibliography{sample}
	\bibliographystyle{amsalpha}
\end{document}